\documentclass[11pt,reqno]{amsart}
\usepackage{amssymb,mathrsfs,graphicx, subfigure, enumerate}
\usepackage{amsmath,amsfonts,amssymb,amscd,amsthm,bbm}
\usepackage{graphicx,colortbl,here}
\usepackage{extpfeil}
\usepackage{subcaption}
\usepackage{caption}
\usepackage{kotex}
\usepackage{hyperref}
\usepackage{cite}
\newcommand{\bbr}{\mathbb R}

\newcommand{\bx}{\mbox{\boldmath $x$}}
\newcommand{\bw}{\mbox{\boldmath $w$}}

\newcommand{\by}{\mbox{\boldmath $y$}}

\newcommand{\bv}{\mbox{\boldmath $v$}}
\newcommand{\bz}{\mbox{\boldmath $z$}}

\newcommand*\di{\mathop{}\!\mathrm{d}}
\usepackage{cite}
\usepackage{upgreek}
\hypersetup{
	hyperindex=true
}

\newcommand{\opnorm}[1]{{ \vert\kern-0.25ex \vert\kern-0.25ex \vert #1 
   \vert\kern-0.25ex \vert\kern-0.25ex \vert}}

\newtheorem{theorem}{Theorem}[section]
\newtheorem{lemma}{Lemma}[section]
\newtheorem{corollary}{Corollary}[section]
\newtheorem{proposition}{Proposition}[section]
\newtheorem{remark}{Remark}[section]
\newtheorem{definition}{Definition}[section]

\hypersetup{
	hyperindex=true
}
\hypersetup{hidelinks}
\allowdisplaybreaks[4]
\begin{document}

\title[Weak flocking behavior of the RKCS model]{Emergent dynamics of spatially extended relativistic kinetic Cucker-Smale model}

\author[Ha]{Seung-Yeal Ha}
\address[Seung-Yeal Ha]{\newline Department of Mathematical Sciences and Research Institute of Mathematics, \newline
	Seoul National University, Seoul, 08826, Republic of Korea}
\email{syha@snu.ac.kr}

\author[Wang]{Xinyu Wang}
\address[Xinyu Wang]{\newline Department of Mathematical Sciences, \newline
	Seoul National University, Seoul, 08826, Republic of Korea. }
\email{wangxinyu97@snu.ac.kr}

\thanks{\textbf{Acknowledgment.} The work of S.-Y. Ha is supported by National Research Foundation(NRF) grant funded by the Korea government(MIST) (RS-2025-00514472) and the work of X. Wang was supported by the Natural Science Foundation of China (grants 123B2003) and Heilongjiang Province Postdoctoral Funding (grants  LBH-Z24167).}

\begin{abstract}
We study the emergent dynamics of the relativistic kinetic Cucker-Smale (RKCS) model without assuming compactness in spatial and velocity support. In this setting, the lower bound of the kernel function in the nonlocal velocity alignment force can be zero so that the previous approach based on the energy method does not provide a quantitative flocking estimate. To overcome this difficulty, we introduce a suitable decay ansatz for the one-particle distribution function and an {\it effective domain} by identifying a time-varying region in which the total mass outside of it decays to zero asymptotically. Using these two ingredients, we show that weak flocking dynamics emerges asymptotically in the sense that the second moment for the velocity fluctuation around the velocity average tends to zero asymptotically, whereas the second moment for spatial fluctuations around the center of mass remains bounded uniformly in time. Our results demonstrate the robustness of the emergent dynamics in the RKCS model across various non-compact physically important distributions, including Gaussian, sub-Gaussian, and $D$-th moment integrable distributions.
\end{abstract}



\maketitle


\section{Introduction}\label{sec:1}
\setcounter{equation}{0}
Collective behaviors of biological complex systems are often observed in nature, to name a few, aggregations of bacteria \cite{Topaz},  flocking of birds \cite{Cucker, cucker2}, swarming of fish \cite{D-M2, Toner}, synchronization of fireflies \cite{Buck},  etc. However, despite its ubiquity, their underlying mechanisms have remained as a mystery for a while, and they 
have long intrigued scientists from diverse disciplines such as biology, control theory,  mathematics, and physics. Recently thanks to engineering applications of collective dynamics in the cooperative control of drones and unmanned aerial vehicles, they received a lot of attention via mathematical modeling. The first generation of mathematical models for synchronization was proposed by Arthur Winfree \cite{Winfree} and Yoshiki Kuramoto  \cite{HK, Kuramoto} almost a half-century ago, and later their modeling spirit was further inherited by the second generation of models by  Reynolds \cite{Re}, Vicsek et al \cite{Vicsek} and Cucker-Smale \cite{Cucker} for flocking. Since then, many phenomenological models have been proposed and extensively studied in the literature.  We refer to survey articles \cite{Al, Vicsek} for a crash introduction to collective dynamics. 

In this work, we consider a physical setting in which a particle's velocity is comparable to the speed of light so that the relativistic effect cannot be ignored, for example, space exploration missions including ESA’s Darwin project \cite{Perea2009}.  In \cite{Ha2017}, Ha and Ruggeri observed a structural analogy between a multi-temperature gas mixture system \cite{Ruggeri2015,Ruggeri2021} and the Cucker-Smale model (CS model). Via the extension of this observation to the relativistic regime, Ha, Kim, and Ruggeri \cite{Ha2020} proposed a relativistic particle model that incorporates multi-temperatures and explicit production terms. By leveraging the techniques from main-field theory \cite{Ruggeri1981} and principal subsystem theory \cite{Boillat1997}, developed by Ruggeri and his collaborators over three decades ago, they further proposed a “{\it mechanical} ” relativistic model as the possible relativistic extension of the CS model. Furthermore, they also demonstrated that their relativistic model reduces to the classical CS model in the classical limit as the speed of light tends to infinity (see \cite{Ha2021}). To set up the stage, we begin with a brief description of the mechanical counterpart of the relativistic CS model (RCS model).

Let $\bx_i$ and $\bv_i$ be the position and velocity of the $i$-th RCS particle, and  we set $\Gamma$ to be the Lorentz factor associated with the velocity $\bv$:
\begin{equation} \label{A-0-0}
\Gamma(\bv)  := \frac{1}{\sqrt{1 - \frac{|\mathbf{\bv}|^2}{c^2}}} = \frac{c}{\sqrt{c^2 - |\mathbf{\bv}|^2}} , \quad \Gamma_i := \Gamma(\bv_i), \quad \forall~ i \in [N]:= \{1, \cdots, N \},
\end{equation}
where $c$ is the speed of light and $| \cdot |$ denotes the standard $\ell^2$-norm in Euclidean space. For notational simplicity, we also introduce new momentum-like observable $\bw$ and a scalar factor $F$ associated with $\bv$:
\begin{equation} \label{A-0}
\bw := F \bv, \quad F := \Gamma \left( 1 + \frac{\Gamma}{c^2}\right), \quad  F_i := F(\bv_i), \quad i \in [N].
\end{equation}
Consider the map $G: [0, \infty) \to [0, \infty)$ defined by 
\begin{equation} \label{A-1}
G:~|\bv| \mapsto |\bw| = F(| \bv|) |\bv| = \frac{c}{\sqrt{c^2 - |\bv|^2}}\left(1 + \frac{1}{c\sqrt{c^2 - |\bv|^2}}\right) |\bv|. 
\end{equation}
Then, it is easy to see that $G$ is an increasing function of $|\bv|$. Hence, the inverse of $G$ exists, and it is also an increasing function. Moreover, it satisfies the following relation via \eqref{A-0}:
\begin{equation} \label{A-2}
G^{-1}(| \bw |) = |\bv|, \quad \mbox{i.e.,} \quad |\bw| \mapsto |\bv|= \frac{|\bw|}{F(G^{-1}(| \bw|))} =:  \frac{|\bw|}{{\tilde F}(| \bw|)}.
\end{equation}
We set 
\[  \tilde{F}_i := {\tilde F}(| \bw_i|), \quad i \in [N]. \]
Now, we consider the Cauchy problem for the RCS model in \cite{Ha2020, HLR}:
\begin{equation}\label{A-3}
\begin{cases}
\displaystyle\frac{\di \bx_i}{\di t} = \frac{\bw_i}{{\tilde F}_i}, \quad t > 0, \quad i \in [N], \\
\displaystyle\frac{\di \bw_i}{\di t} = 
\frac{\kappa}{N} \sum_{j=1}^N \phi(|\mathbf{\bx}_i - \mathbf{\bx}_j|)  \left( \frac{\bw_j}{{\tilde F}_j} - \frac{\bw_i}{{\tilde F}_i}  \right), \\
\displaystyle (\bx_i, \bw_i)\Big|_{t=0} = (\bx_{i0}, \bw_{i0}),
\end{cases}
\end{equation}
where $\kappa$ is the nonnegative coupling strength, and $\phi:  [0, \infty) \to [0, \infty)$ is the communication weight function defined by 
\begin{equation}\label{A-3-1}
\phi(s)=\frac{1}{\left(1+s^2\right)^{\frac{\beta}{2}}},  \quad \beta \ge 0.
\end{equation}
Note that the speed of the relativistic velocity $\frac{\bw_i}{{\tilde F}_i}$ is bounded by the speed of light:
\begin{equation} \label{A-3-2}
\Big | \frac{\bw_i}{{\tilde F}_i} \Big | \leq c, \quad i\in[N],
\end{equation}
and at the formal level, it follows from \eqref{A-0-0}, \eqref{A-0}, \eqref{A-1} and \eqref{A-3} that  
\begin{equation} \label{A-4}
{\tilde F}_i \to 1, \quad \bw_i \to \bv_i, \quad i \in [N] \quad \mbox{as $c \to \infty$}.
\end{equation}
Thus, system \eqref{A-3} reduces  the Cauchy problem for the CS model:
\begin{equation}  \label{A-5}
\begin{cases} 
\displaystyle	\frac{\di \mathbf{\bx}_i}{\di t} = \mathbf{\bv}_i, \quad t > 0, \quad i \in [N],\\
\displaystyle	\frac{\di \mathbf{\bv}_i}{\di t} = \frac{\kappa}{N} \sum_{j=1}^N \phi(|\mathbf{\bx}_i - \mathbf{\bx}_j|) ( \bv_j - \bv_i), \\
\displaystyle (\bx_i, \bv_i)\Big|_{t=0}= (\bx_{i0}, \bv_{i0}).
\end{cases}
\end{equation}
So far,  the classical CS model \eqref{A-5} has been extensively investigated from various perspectives, e.g., collision avoiding \cite{collision1,collision2}, hierarchical leadership \cite{leader,ex5}, rooted leadership flocking \cite{switch, L12}, multi-cluster flocking \cite{bi-cluster,bi-cluster2,bi-cluster3}, stochastic flocking \cite{ks1,ks2}, time-discrete flocking \cite{DS1,DS2,discreteDHK}, infinite particle \cite{w1,w3,J1,w5}, kinetic model \cite{k2,k1,k6,HZ2018,kinetic3,kinetic2, kinetic4,kinetic1,kinetic5,k4,k12,a2,w4}, etc. 

Next, we consider the mean-field limit equation for the relativistic system \eqref{A-3}. For this, we introduce a one-particle distribution function (or probability density function) $f = f(t, \bx, \bw)$. Then, we use the standard BBGKY Hierarchy argument to see that the dynamics of $f$ is governed by the Cauchy problem for the RKCS model \cite{Ha2021}:
\begin{equation}\label{RKCS}
	\begin{cases}
	\displaystyle	\partial_t f + \frac{\bw}{\tilde{F}} \cdot \nabla_{\mathbf{\bx}} f + \nabla_{\mathbf{\bw}} \cdot ({\bf F}_a[f] f) = 0, \quad (t, \mathbf{\bx}, \mathbf{\bw}) \in \mathbb{R}_+ \times \mathbb{R}^d \times \mathbb{R}^d, \\
		\displaystyle	{\bf F}_a[f](t, \mathbf{\bx}, \bw) := \kappa \int_{\mathbb{R}^{2d}} \phi(|\mathbf{\bx}_\star - \mathbf{\bx}|) \left( \frac{\bw_\star}{{\tilde F}_\star} - \frac{\bw}{{\tilde F}} \right) 
 f(t, \bx_\star, \bw_\star) \di \bx_\star \di \bw_\star, \\
       \displaystyle f \Big|_{t = 0} = f_0.
	\end{cases}
\end{equation}
In this paper, we are interested in the emergent dynamics of the measure-valued solution to \eqref{RKCS} with {\it non-compact velocity-spatial support}. First, we recall the concept of the measure-valued solution to \eqref{RKCS} in the space of probability measures on the phase space $\bbr^{2d}$ denoted by ${\mathcal P}(\bbr^{2d})$ and weak flocking. For $\mu\in L^\infty([0,T);\mathcal{P}(\mathbb{R}^{2d}))$, we set 
\[
\|  \mu_t \| := \int_{\bbr^{2d}} \mu_t(\di \bx, \di \bw), \quad \bx_c(t) := \frac{1}{\|  \mu_t \| } \int_{\mathbb{R}^{2d}} \bx \mu_t( \di \bx, \di \bw), \quad \bw_c(t) := \frac{1}{\|  \mu_t \| } \int_{\mathbb{R}^{2d}} \bw \mu_t( \di \bx, \di \bw).
\]
\begin{definition} \label{D1.1}
\begin{enumerate}
\item
The measure-valued map $\mu\in L^\infty([0,T);\mathcal{P}(\mathbb{R}^{2d}))$ is a measure-valued solution of \eqref{RKCS} with the initial datum $\mu_0 \in\mathcal{P}(\mathbb{R}^{2d})$ if the following relations hold: 
			\begin{enumerate}
			\item
			 $\mu$ is weakly continuous in $t$, i.e.,
			\begin{equation*}
			t \quad \mapsto \quad \int_{\mathbb{R}^{2d}}\psi(\bx,\bw)\mu_t(\di \bx,\di \bw) \quad\text{is continuous}\quad \forall~\psi \in C_0^1(\mathbb{R}^{2d}),
			\end{equation*}
			where we used a handy notation $\mu_t(\di \bx, \di \bw) := \mu(t, \di \bx, \di \bw)$. 
			\vspace{0.2cm}
			\item
			 $\mu$ satisfies \eqref{RKCS} in weak sense:
			\begin{align*}
			\begin{aligned}
			& \int_{\mathbb{R}^{2d}}\psi(\bx, \bw) \mu_t(\di \bx,\di \bw)  = \int_{\mathbb{R}^{2d}}\psi(\bx, \bw) \mu_0(\di \bx,\di \bw)  \\
			& \hspace{1.5cm} + \int_0^t\int_{\mathbb{R}^{2d}} \Big (\partial_s\psi+\frac{\bw}{\tilde{F}}\cdot \nabla_{\bx} \psi+\nabla_{\bw}\psi\cdot {\bf F}_a[\mu_s] \Big ) \Big|_{(\bx, \bw)} \mu_s(\di \bx,\di \bw) \di s,
			\end{aligned}
			\end{align*}
			for $\psi\in C_c^1([0,T)\times\mathbb{R}^{2d})$ and the velocity alignment forcing ${\bf F}_a[\mu_t]$ is given as follows.
\[
{\bf F}_a[\mu_t](t, \mathbf{\bx}, \bw) := \kappa \int_{\mathbb{R}^{2d}} \phi(|\mathbf{\bx}_\star - \mathbf{\bx}|) \left( \frac{\bw_\star}{{\tilde F}_\star} - \frac{\bw}{{\tilde F}} \right) 
 \mu_t(\di \bx_\star, \di \bw_\star).
\]	
\end{enumerate}
\item		
The Cauchy problem \eqref{RKCS}  exhibits weak (mono-cluster) flocking if the measure-valued solution  $\mu\in L^\infty([0,T);\mathcal{P}(\mathbb{R}^{2d}))$ for \eqref{RKCS} satisfies 
\[
\begin{cases}
\displaystyle \sup\limits_{0\le t< \infty} \int_{\mathbb{R}^{2d}} |\bx - \bx_c(t) |^2 \mu_t(\di \bx, \di \bw)<\infty:~&~\mbox{spatial cohesion}, \vspace{0.2cm}\\
\displaystyle \lim_{t \to \infty} \int_{\mathbb{R}^{2d}} |\bw - \bw_c(t)|^2 \mu_t(\di \bx, \di  \bw)  = 0:~~&\mbox{velocity alignment}.
\end{cases}
\]
\end{enumerate}
\end{definition}
\noindent In \cite{Ha2020, Ha2021}, the authors focus solely on compactly supported initial datum, which restricts the applicability of system \eqref{RKCS} in a more general setting, like Cauchy (Lorentz), Gaussian, sub-Gaussian distributions and so on. Recently, the authors \cite{w2,w7} studied the asymptotic behavior of the second-order velocity moment of the KCS model and RKCS model, respectively. However, compared to the velocity moment, the position moment can reflect the degree of aggregation. The collective behavior of the position support has been extensively investigated for the KCS model from various perspectives \cite{k6,w6,w8}, while the emergent behavior of the position support of the RKCS model with non-compact velocity-spatial support remains unexplored so far. 
Thus, a natural question is the following emergent dynamics for the RKCS model \eqref{RKCS}: \vspace{0.1cm}

\begin{center}
	``Under what conditions on system parameters and initial datum with non-compact support, can we derive a weak flocking dynamics for \eqref{RKCS}?"
\end{center}
\vspace{0.1cm}

In this paper, we answer the above question in an affirmative manner. More precisely, our main results of this paper can be summarized as follows. First, we consider the class of polynomially decaying initial distributions with unit mass:
\[
 \int_{\mathbb{R}^{2d}}  \mu_0(\di \bx,\di \bw) = 1, \quad {\mathcal M}_p(\mu_0, D) := \int_{\mathbb{R}^{2d}}( |\bx|^D+ |\bw|^D )\mu_0(\di \bx,\di \bw) < \infty, \quad D \gg 1,
\]
where the subscript $p$ denotes the initial for a polynomial decay.  These finite moments are propagated along the RKCS flow \eqref{RKCS} (see Lemma \ref{L2.2} and Lemma \ref{L2.4}):
\[
 \int_{\mathbb{R}^{2d}}  \mu_t(\di \bx,\di \bw) = 1, \quad  \int_{\mathbb{R}^{2d}}  |\bw|^D \mu_t(\di \bx,\di \bw) \leq  {\mathcal M}_p(\mu_0, D), \quad 
  \int_{\mathbb{R}^{2d}} |\bx|^D \mu_t(\di \bx,\di \bw) \lesssim (1 + t)^D,
\]
where $f \lesssim g$ means that there exists a generic constant $C$ independent of $f$ and $g$ such that $f \leq C g$. Then we derive polynomial decay estimates: there exists a positive constant $\lambda > 2$ such that 
\begin{align*}
\begin{aligned}
& \int_{\mathbb{R}^{2d}} |\bw - \bw_c(t)|^2 \mu_t(\di \bx, \di  \bw) \lesssim (1 + t)^{-\lambda} \quad \mbox{and} \\
& \Big| \left(\int_{\mathbb{R}^{2d}} |\bx -\bx_c (t)|^2\mu_t(\di \bx,\di \bw)\right)^{\frac{1}{2}} - \left(\int_{\mathbb{R}^{2d}} |\bx -\bx_c(0) |^2\mu_0(\di \bx,\di \bw)\right)^{\frac{1}{2}} \Big|  \lesssim \Big(1+(1 + t)^{-(\frac{\lambda}{2} - 1)} \Big).
\end{aligned}
\end{align*}
For the details, we refer to Theorem \ref{T3.1} and Section \ref{sec:4}. Second, we consider the class of exponentially decaying initial distributions:
\[  \int_{\mathbb{R}^{2d}}  \mu_0(\di \bx,\di \bw) = 1, \quad {\mathcal M}_e(\mu_t, \alpha) = \int_{\mathbb{R}^{2d}} e^{\alpha\left(|\bx|+ |\bw|\right)} \mu_t(\mathrm{\di}\bx, \mathrm{\di }\bw) < \infty,\quad \text{for} \quad \alpha>0. \]
Similar to the case of polynomially decaying distributions, there exists a positive constant $\tilde{\lambda} > 0$ such that 
\[
\int_{\mathbb{R}^{2d}} |\bw - \bw_c(t)|^2 \mu_t(\di \bx, \di  \bw) \lesssim \exp \Big(-(1 + t)^{-{\tilde \lambda}} \Big), \quad \sup_{0 \leq t < \infty} \int_{\mathbb{R}^{2d}} |\bx-\bx_c(t) |^2\mu_t(\di \bx,\di \bw) < \infty.
\]
For the details, we refer to Theorem \ref{T3.2} and Section \ref{sec:5}. 

\vspace{0.3cm}

The rest of this paper is organized as follows. In Section \ref{sec:2}, we first review basic definitions and a priori estimates for the RKCS model. In Section \ref{sec:3}, we summarize our main results and a proof strategy. In Section \ref{sec:4} and Section \ref{sec:5}, we provide proofs of our main results for algebraically and exponentially decaying distributions, respectively. Finally, Section \ref{sec:6} is devoted to a brief summary of our main results and some discussions for a future work.

\vspace{0.3cm}

\noindent\textbf{Notation:} We define the \(\bx\)-support, \(\bw\)-support and their diameters of measure-valued solution $f$ as follows:
\begin{align*}
\begin{aligned}
&\Omega_{\bx}(f(t)) := \mathbb{P}_{\bx}\operatorname{supp} f(t, \cdot, \cdot), \qquad\quad \Omega_{\bw}(f(t)) := \mathbb{P}_{\bw}\operatorname{supp} f(t, \cdot, \cdot),\\
& D_{\bw}(t):= \sup_{\bw, \bw_{\star} \in \Omega_{w}(f(t))} |\bw_{\star}- \bw|,\quad D_{\bx}(t) := \sup_{\bx, \bx_{\star}\in \Omega_x(f(t))} |\bx_{\star} - \bx|,\\
\end{aligned}
\end{align*}
where $\mathbb{P}_{\bx}$ and $\mathbb{P}_{\bw}$ denote the projection on $\mathbb{R}_{\bx}^d$ and $\mathbb{R}_{\bw}^d$, respectively, and we set
\[ \bbr = (-\infty, \infty), \quad \bbr_+ = (0, \infty). \]

\section{Preliminaries}\label{sec:2}
\setcounter{equation}{0}
In this section, we recall basic materials and study basic estimates for later sections. 
	\subsection{Basic materials}\label{sec:2.1}
	In this subsection, we study several preparatory materials to be used throughout the paper.  First, we recall several definitions on measure and distance between them. 
\begin{definition}\label{def:23}
Let $T\in (0,\infty]$ be a positive real number.
\begin{enumerate}
\item
The support of a Borel measure $\mu$ on $\mathbb{R}^{2d}$ is the closure of the set $\{(\bx,\bw)\in\mathbb{R}^{2d}:\mu(B_r(\bx,\bw))>0, \forall\ r>0\}$.
\vspace{0.1cm}
\item
Let $\mu$ be a Borel measure on $\mathbb{R}^{n}$, and let $T: \mathbb{R}^{n}\rightarrow\mathbb{R}^{n}$ be a measurable map. The push-forward measure of $\mu$ by $T$ is the measure $T\#\mu$ defined by 
\[ T\#\mu(O)=\mu(T^{-1}(O)), \quad \mbox{for all Borel set $O\subset\mathbb{R}^{n}$}. \]
\end{enumerate}
\end{definition}
The definition of Wasserstein distance is as follows.
		\begin{definition}\label{def:24}
		For $p\in[1,\infty)$, let $\mu$ and $\nu$ be two probability measures in $\mathcal{P}(\mathbb{R}^{d})$. Then, the Wasserstein distance $W_p(\mu,\nu)$ of order $p$ between $\mu$ and $\nu$ is defined as follows.
			\begin{equation*}
				W_p(\mu,\nu):=\inf\left\{\left(\int_{\mathbb{R}^{2d}} |\bx-\bx_{\star}|^p\pi(\di \bx,\di \bx_{\star})\right)^\frac{1}{p}:\pi\in\Pi(\mu,\nu)\right\},
			\end{equation*}
			where $\Pi(\mu,\nu)$ is the set of joint probability measures on $\mathbb{R}^{2d}$ with marginals $\mu$ and $\nu$, respectively. For all integrable measurable functions $\psi$, ${\tilde \psi}$ on $\mathbb{R}^{d}$,
			\begin{equation}\label{f3.1}
				\int_{\mathbb{R}^{d}\times\mathbb{R}^{d}}(\psi(\bx)+{\tilde \psi}(\bx_{\star}))\pi(\di \bx,\di \bx_{\star})=\int_{\mathbb{R}^{d}}\psi(\bx)\mu(\di \bx)+\int_{\mathbb{R}^{d}}{\tilde \psi}(\bx_{\star})\nu(\di \bx_{\star}).
			\end{equation}
		\end{definition}
\begin{remark}		
In order to avoid the trouble that $W_p$ may take the value $+\infty$, we consider $W_p$ on $\mathcal{P}_p(\mathbb{R}^{d})$:
		\begin{equation*}
			\mathcal{P}_p(\mathbb{R}^{d}):=\left\{\mu\in\mathcal{P}(\mathbb{R}^{d}):\int_{\mathbb{R}^{d}}|\bz|^p\mu(\di \bz)<\infty\right\}.
		\end{equation*}
\end{remark}	
In the next lemma, we characterize the convergence in $W_p$.
		\begin{lemma}
		\emph{\cite{k8}} \label{le2.5}
The following assertions hold. 
\begin{enumerate}
\item		
The space $\mathcal{P}_p(\mathbb{R}^{d})$ endowed with the p-Wasserstein metric is a complete metric space.
\item			
Let $(\mu_k)_{k\in\mathbb{N}}$ be a sequence of probability measures in $\mathcal{P}_p(\mathbb{R}^{d})$, and let $\mu$ be an element of $\mathcal{P}_p(\mathbb{R}^{d})$. We say that $(\mu_k)_{k\in\mathbb{N}}$ has uniformly integrable p-th moments if for some $x_0\in\mathbb{R}^d$$:$
			\begin{equation*}
				\lim\limits_{r\rightarrow\infty}\int_{\mathbb{R}^d\setminus B_r(\bx_0)} |\bx-\bx_0|^p\mu_k(\di \bx)=0\quad \text{uniformly with respect to }k\in\mathbb{N}.
			\end{equation*}
			Moreover, we say that $\mu_k$ converge weakly to $\mu$ if
			\begin{align}\label{e3.1}
				\lim\limits_{k\rightarrow\infty}\int_{\mathbb{R}^d}\psi(\bx) \mu_k(\di \bx)=\int_{\mathbb{R}^d}\psi(\bx) \mu(\di \bx)\quad \text{for all}\quad\psi \in C_b(\mathbb{R}^d).
			\end{align}
			In particular, we have 
			\begin{equation*}
				\lim\limits_{k\rightarrow\infty}W_p(\mu_k,\mu)=0 \quad \Longleftrightarrow \quad \left\{\begin{aligned}
					&\mu_k\quad\text{converge weakly to}\quad\mu,\\
					&(\mu_k)_{k\in\mathbb{N}} \quad\text{has uniformly integrable p-moments.}
				\end{aligned}\right.
			\end{equation*}
\end{enumerate}			
		\end{lemma}
\subsection{Elementary estimates} \label{sec:2.2} 
In this subsection, we study several a priori estimates. In the following lemmas, we recall several estimates from earlier works \cite{Ha2020, Ha2021, w7}. Without loss of generality, we assume that $\kappa=1$ in the sequel.
\begin{lemma} \label{L2.2}
\emph{\cite{Ha2020,Ha2021,w7}} 
Suppose that $\mu_0 \in P_p(\mathbb{R}^{2d}) (p\ge2)$ is the nonnegative initial probability measure. Then, there exists a unique global measure-valued solution $\mu \in L^{\infty}([0,T);\mathcal{P}_p(\mathbb{R}^{2d}))$ to \eqref{RKCS} with the following moment estimates:~for $t > 0$, 
		\begin{align*}
		\begin{aligned}
	         & (i)~ \frac{\di}{\di t}\int_{\bbr^{2d}}\mu_t(\di \bx,\di \bw)=0, \quad \frac{\di}{\di t}\int_{\bbr^{2d}}\bw \mu_t(\di \bx,\di \bw)=0. \\
	        &  (ii)~\dfrac{\di}{\di t}\int_{\mathbb{R}^{2d}} |\bw|^2\mu(t,\di \bx,\di \bw)\\ 
	        & \hspace{1cm} =-\int_{\bbr^{4d}}\left\langle \bw_{\star}-\bw,\frac{\bw_{\star}}{{\tilde F}_{\star}} - \frac{\bw}{{\tilde F}} \right\rangle   \phi(|\bx-\bx_\star|)\mu_t(\di \bx,\di \bw) \mu_t(\di \bx_{\star},\di \bw_{\star})\le 0. \\
	         & (iii)~ \dfrac{\di}{\di t}\int_{\bbr^{2d}} |\bw-\bw_c(t)|^2\mu_t(\di \bx,\di \bw) \\
	         &  \hspace{1cm} =-\int_{\bbr^{4d}}\left\langle \bw_{\star}-\bw,\frac{\bw_{\star}}{{\tilde F}_{\star}} - \frac{\bw}{{\tilde F}} \right\rangle   \phi(|\bx-\bx_\star|)\mu_t(\di \bx,\di \bw) \mu_t(\di \bx_{\star},\di \bw_{\star})\le 0.
		 \end{aligned}
	\end{align*}
\end{lemma}	
\begin{remark} \label{R2.1}
It follows from the first assertion of Lemma \ref{L2.2} that 
\[ \| \mu_t \| = \| \mu_0 \|, \quad \bw_c(t) = \bw_c(0), \quad \forall~t > 0. \]
\end{remark}
\begin{lemma}  
\emph{\cite{Ha2020,Ha2021}} \label{L2.3}
Suppose that \( \bx, \by \in \mathbb{R}^d \) satisfy
	\[
	|\bx| \le R, \quad |\by| \le R, \quad \mbox{for some $R > 0$}.
	\]
Then, there exists a positive constant $\Lambda(R)$ such that 
\[
(\bx - \by) \cdot \left( \frac{\bx}{\tilde{F}(\bx)} - \frac{\by}{\tilde{F}(\by) } \right) \geq \Lambda(R) |\bx - \by|^2, \quad \left|\frac{\bx}{\sqrt{1 + \frac{|\bx|^2}{c^2}}} - \frac{\by}{\sqrt{1 + \frac{|\by|^2}{c^2}}} \right|\le |\bx-\by|.
\]
\end{lemma}
\begin{proof}
For detailed proofs, we refer to \cite{Ha2020,Ha2021}.
\end{proof}
\begin{remark} \label{R2.3}  By direct calculation, the optimal order of positive constant $\Lambda(R)$ with respect to $ R$ is explicitly given by 
\[ \Lambda(R) :=\mathcal{O}\left(\frac{1}{2\left(1+\frac{R^2}{c^2}\right)^{\frac{3}{2}}}\right). \] 
\end{remark}
Recall that 
\[ {\mathcal M}_p(\mu_0, D) = \int_{\mathbb{R}^{2d}}( |\bx|^D+ |\bw|^D )\mu_0(\di \bx,\di \bw),  \quad {\mathcal M}_e(\mu_0, \alpha):= \int_{\mathbb{R}^{2d}}e^{\alpha(|\bx| + |\bw|)} \mu_0(\di \bx,\di \bw).
\]
\begin{lemma} \label{L2.4}
For $D \in {\mathbb N}$ and $T \in (0, \infty]$, let $\mu \in L^{\infty}([0,T);\mathcal{P}_p(\mathbb{R}^{2d}))$ be a measure-valued solution to \eqref{RKCS} with the initial datum $\mu_0$ satisfying 
\[
 \int_{\mathbb{R}^{2d}}  \mu_0(\di \bx,\di \bw) = 1, \quad  {\mathcal M}_p(\mu_0, D)  < \infty.
\]
Then, the following estimates hold: for $t \geq 0$,
\begin{align*}
\begin{aligned}
& (i)~\int_{\mathbb{R}^{2d}} |\bw|^{D}\mu_t(\di \bx,\di \bw) \le	\int_{\mathbb{R}^{2d}} |\bw|^{D}\mu_0( \di \bx,\di \bw). \\
& (ii)~\int_{\mathbb{R}^{2d}} |\bx|^{D}\mu_t(\di \bx,\di \bw) \le	\Big[ \left(\int_{\mathbb{R}^{2d}}|\bx|^{D}\mu_0( \di \bx,\di \bw)\right)^{\frac{1}{D}}+ ct \Big ]^{D}.
\end{aligned}
\end{align*}
\end{lemma}
\begin{proof}
It follows from the defining relation of measure-valued solution that for any $\psi\in C_c^1(\mathbb{R}^{2d})$, we have
	\begin{equation}\label{B-0-0}
		\dfrac{\di}{\di t}\int_{\mathbb{R}^{2d}}\psi \mu_t(\di \bx,\di \bw)=\int_{\mathbb{R}^{2d}}\left\{\frac{\bw}{{\tilde F}} \cdot\nabla_{\bx}\psi +\nabla_{\bw}\psi\cdot {\bf F}_a[\mu_t]\right\}\mu_t(\di\bx,\di \bw).
	\end{equation}
In the sequel, we use $C^{\infty}$ functions $\bw \mapsto |\bw|^D$ and $\bx \mapsto |\bx|^D$ as test functions, but they are not compactly supported. Hence, as it is, they cannot be used as test functions, but this can be done using a smooth cut-off function. More precisely, we introduce a smooth cut-off function \(\chi_R\) satisfying 
	\[
	\chi_R(\bz) = 
	\begin{cases} 
		1, & \text{for } \bz \in B_R, \\ 
		0, & \text{for } \bz\in \mathbb{R}^d \setminus B_{2R},
	\end{cases}
	\quad  |\nabla \chi_R| \leq \frac{2}{R}.
\]
Then, we substitute $\psi(\bw) = \chi_R(\bx) \chi_R(\bw) |\bw|^D$ into \eqref{B-0-0} and use  real analysis techniques and $R \to \infty$ to recover all the desired estimates as we want. Since this procedure is very standard, we omit its details. Thus, from now on, we use the maps $\bw \mapsto |\bw|^D$ and $\bx \mapsto |\bx|^D$ as test functions in \eqref{B-0-0}. 
	
\vspace{0.2cm}

\noindent (i)~We substitute $\psi(\bw) =|\bw|^D$  into \eqref{B-0-0} to get  the first desired estimate:
		\begin{align}
			\begin{aligned} \label{NNN-1}
		&\dfrac{\di}{\di t}\int_{\mathbb{R}^{2d}} |\bw|^D \mu_t(\di \bx,\di \bw)\\
		& \hspace{0.5cm} =-D\int_{\mathbb{R}^{4d}} |\bw|^{D-2} \Big \langle\bw,\frac{\bw}{{\tilde F}}-\frac{\bw_{\star}}{{\tilde F}_{\star}} \Big \rangle\mu_t(\di\bx_{\star},\di \bw_{\star})\mu_t(\di\bx,\di \bw)\\
			& \hspace{0.5cm} =-\frac{D}{2}\int_{\mathbb{R}^{4d}} \Big \langle |\bw|^{D-2}\bw-|\bw_{\star}|^{D-2}\bw_{\star},\frac{\bw}{{\tilde F}}-\frac{\bw_{\star}}{{\tilde F}_{\star}} \Big \rangle\mu_t(\di \bx_{\star}, \di \bw_{\star})\mu_t(\di\bx,\di \bw)\\
			& \hspace{0.5cm} \le-\frac{D}{2}\int_{\mathbb{R}^{4d}} \Big( \frac{|\bw|^{D}}{{\tilde F}}+\frac{|\bw_{\star}|^{D}}{{\tilde F}_{\star}}- |\bw_{\star}|^{D-1}\frac{|\bw|}{{\tilde F}}-|\bw|^{D-1}\frac{|\bw_{\star}|}{{\tilde F}_{\star}} \Big) \mu_t(\di\bx_{\star},\di \bw_{\star})\mu_t(\di\bx,\di \bw)\\
			& \hspace{0.5cm} =-\frac{D}{2}\int_{\mathbb{R}^{4d}}(|\bw|^{D-1}-|\bw_{\star}|^{D-1})\left(\frac{|\bw|}{{\tilde F}}-\frac{|\bw_{\star}|}{{\tilde F}_{\star}}\right)\mu_t(\di\bx_{\star},\di \bw_{\star})\mu_t(\di\bx,\di \bw)\\
			& \hspace{0.5cm} \le 0,
			\end{aligned}
	\end{align}
where we use the increasing property of $\frac{|\bw|}{{\tilde F}}$ in the last inequality of \eqref{NNN-1}:
\[
(|\bw|^{D-1}-|\bw_{\star}|^{D-1})\left(\frac{|\bw|}{{\tilde F}}-\frac{|\bw_{\star}|}{{\tilde F}_{\star}}\right) \ge 0.
\]
\vspace{0.2cm}

\noindent (ii)~We use $\eqref{RKCS}_1$, the H\"{o}lder inequality and $ \Big| \frac{\bw}{\tilde{F}} \Big| \leq c$ to get 
		\begin{align*}
		\begin{aligned}
		& \dfrac{\di}{\di t}\int_{\mathbb{R}^{2d}} |\bx|^D \mu_t(\di \bx, \di \bw) =D\int_{\mathbb{R}^{4d}} |\bx|^{D-2} \Big \langle \bx, \frac{\bw}{\tilde{F}} \Big \rangle \mu_t(\di \bx, \di \bw) \le c D\int_{\mathbb{R}^{4d}} | \bx|^{D-1}  \mu_t(\di \bx, \di \bw) \\
		& \hspace{0.5cm} \le cD\left(\int_{\mathbb{R}^{4d}} |\bx|^{D} \mu_t(\di \bx, \di \bw) \right)^{\frac{D-1}{D}}\left(\int_{\mathbb{R}^{4d}}  \mu_t(\di \bx, \di \bw) \right)^{\frac{1}{D}} =  cD\left(\int_{\mathbb{R}^{4d}} |\bx|^{D} \mu_t(\di \bx, \di \bw) \right)^{\frac{D-1}{D}}.
		\end{aligned}
	\end{align*}
This implies the desired second assertion:
\[
\Big( \int_{\mathbb{R}^{4d}} |\bx|^{D} \mu_t(\di \bx, \di \bw) \Big)^{\frac{1}{D}} \leq \Big( \int_{\mathbb{R}^{4d}} |\bx|^{D} \mu_0(\di \bx, \di \bw) \Big)^{\frac{1}{D}} + ct,
\]	
i.e.,
\[
 \int_{\mathbb{R}^{4d}} |\bx|^{D} \mu_t(\di \bx, \di \bw) \leq \Big[     \Big( \int_{\mathbb{R}^{4d}} |\bx|^{D} \mu_0(\di \bx, \di \bw) \Big)^{\frac{1}{D}} + ct \Big]^D.                                
\]
\end{proof}
For notational simplicity, we use the following abbreviated notation in what follows:
\begin{align*}
\begin{aligned}
& \int_{|\bx| \geq R_x} (\cdots) \mu_t(\di \bx, \di \bw)  := \int_{\bbr^d}  \int_{B^c_{R_x}} (\cdots) \mu_t(\di \bx, \di \bw), \\
&  \int_{|\bw| \geq R_w} (\cdots) \mu_t(\di \bx, \di \bw)  := \int_{\bbr^d} \int_{ B^c_{R_w}} (\cdots) \mu_t(\di \bx, \di \bw).
\end{aligned}
\end{align*}

As a direct application of Lemma \ref{L2.4}, we have estimates on spatial and velocity moments for large position and velocity.
\begin{corollary} \label{C2.1}
For given constants $D \in {\mathbb N}$ and $T \in (0, \infty]$, let $\mu \in L^{\infty}([0,T);\mathcal{P}_D(\mathbb{R}^{2d}))$ be a measure-valued solution to \eqref{RKCS} with the initial datum $\mu_0$ satisfying
\[
 \int_{\mathbb{R}^{2d}}  \mu_0(\di \bx,\di \bw) = 1, \quad {\mathcal M}_p(\mu_0, D) < \infty,
\]
and let $R_x(t)$ and $R_w(t)$ be nonnegative locally bounded functions. Then, the following estimates hold:~for $t \geq 0$,
\begin{align}
\begin{aligned} \label{B-0-1}
& (i)~ \int_{|\bw| \geq R_{w}(t)} \mu_t(\di \bx, \di \bw)  \le \frac{\mathcal{M}_{p}(\mu_0,D)}{R_w(t)^D}. \\
& (ii)~ \int_{|\bx| \geq R_{x}(t)} \mu_t(\di \bx, \di \bw) \le \frac{   2^{D-1} \max \Big \{  {\mathcal M}_p(\mu_0, D), c^D \Big\} (1 + t)^D }{R_x(t)^D}.
\end{aligned}
\end{align}
\end{corollary}
\begin{proof}
Let $R_x$ and $R_w$ be arbitrary time-dependent functions in $t$. Then, it follows from Lemma \ref{L2.4} that 
\begin{align}\label{B-0-1-1}
\begin{aligned}
& R_w(t)^D \int_{|\bw| \geq R_w(t)} \mu_t(\di \bx, \di \bw)   \le \int_{\mathbb{R}^{2d}} |\bw|^D \mu_t(\di \bx, \di \bw) \leq \int_{\mathbb{R}^{2d}} |\bw|^{D}\mu_0( \di \bx,\di \bw) = \mathcal{M}_{p}(\mu_0, D),\\
& R_x(t)^D  \int_{|\bx| \geq R_x(t)}  \mu_t(\di \bx, \di \bw)   \di z \leq  \int_{|\bx| \geq R_x(t)} |\bx|^D  \mu_t(\di \bx, \di \bw) \leq  \int_{\bbr^{2d}} |\bx|^D  \mu_t(\di \bx, \di \bw)    \\
& \hspace{0.5cm}  \le \Big[ \left(\int_{\mathbb{R}^{2d}}|\bx|^{D}\mu_0( \di \bx,\di \bw)\right)^{\frac{1}{D}}+ ct \Big ]^{D} \le  2^{D-1} \Big[ \int_{\mathbb{R}^{2d}}|\bx|^{D}\mu_0( \di \bx,\di \bw) + c^D t^D \Big ] \\
& \hspace{0.5cm} \leq 2^{D-1} \max \Big \{  {\mathcal M}_p(\mu_0, D), c^D\Big \} (1 + t)^D,
\end{aligned}
\end{align}
where we used the inequality:
\[ (|a| + |b|)^{D} \leq 2^{D-1} (|a|^D + |b|^D),~~D \geq 1. \]
The above estimates in \eqref{B-0-1-1} yield the desired estimates. 
\end{proof}
Next, we provide the estimate of exponential moments of the solution to system \eqref{RKCS}.
\begin{lemma} \label{L2.5}
For $\alpha > 0$ and $T \in (0, \infty]$, let $\mu \in L^{\infty}([0,T);\mathcal{P}(\mathbb{R}^{2d}))$ be a measure-valued solution to \eqref{RKCS} with the initial datum $\mu_0$ satisfying
\[  \int_{\mathbb{R}^{2d}}  \mu_0(\di \bx,\di \bw) = 1, \quad  {\mathcal M}_e(\mu_0, \alpha) < \infty. \]
Then, the following estimates hold:~for $t \geq 0$,
\begin{align*}
\begin{aligned}
& (i)~\int_{\mathbb{R}^{2d}}e^{\alpha |\bw|} \mu_t(\di \bx,\di \bw) \le  \int_{\mathbb{R}^{2d}}e^{\alpha |\bw|} \mu_0(\di \bx,\di \bw) . \\
& (ii)~\int_{\mathbb{R}^{2d}}e^{\alpha |\bx|} \mu_t(\di \bx,\di \bw) \le \Big( \int_{\mathbb{R}^{2d}}e^{\alpha |\bx|} \mu_0(\di \bx,\di \bw) \Big) e^{c\alpha t}.
\end{aligned}
\end{align*}
\end{lemma}
\begin{proof} Similar to the proof of Lemma \ref{L2.4}, we derive:
\begin{align}
\begin{aligned} \label{B-0-2}
&\dfrac{\di}{\di t} \int_{\mathbb{R}^{2d}}e^{\alpha |\bw|} \mu_t(\di \bx,\di \bw) \\
& \hspace{0.5cm} = \alpha \int_{\bbr^{2d}}e^{\alpha |\bw|} \Big \langle \frac{\bw}{|\bw|}, {\bf F}_a[\mu_t]  \Big \rangle \mu_t(\di \bx,\di \bw)   \\
&  \hspace{0.5cm} = \alpha   \int_{\bbr^{4d}} e^{\alpha |\bw|} \phi(|\bx_\star - \bx|)  \Big \langle \frac{\bw}{|\bw|}, \frac{\bw_\star}{{\tilde F}_\star} - \frac{\bw}{{\tilde F}}  \Big \rangle \mu_t(\di \bx_\star,\di \bw_\star) \mu_t(\di \bx,\di \bw)  \\
& \hspace{0.5cm} =-\frac{\alpha}{2}   \int_{\bbr^{4d}} \phi(|\bx_\star - \bx|)  \Big \langle e^{\alpha |\bw_\star|}\frac{\bw_{\star}}{|\bw_{\star}|} -e^{\alpha |\bw|} \frac{\bw}{|\bw|} , \frac{\bw_\star}{{\tilde F}_\star} - \frac{\bw}{{\tilde F}}  \Big \rangle \mu_t(\di \bx_\star,\di \bw_\star) \mu_t(\di \bx,\di \bw) \\
& \hspace{0.5cm} \le0,
\end{aligned}
\end{align}
where we use rough estimates:
\[  \Big \langle e^{\alpha |\bw_\star|}\frac{\bw_{\star}}{|\bw_{\star}|} -e^{\alpha |\bw|} \frac{\bw}{|\bw|} , \frac{\bw_\star}{{\tilde F}_\star} - \frac{\bw}{{\tilde F}}  \Big \rangle\geq\left(e^{\alpha |\bw_\star|}-e^{\alpha |\bw|}\right)\left( \frac{|\bw_\star|}{{\tilde F}_\star}- \frac{|\bw|}{{\tilde F}}\right)\ge0.
\]
Then, we apply Grönwall's lemma to get the desired estimate.  \newline

\noindent (ii)~Similar to (i), we use $\psi = e^{\alpha |\bx|}$ as a test function in \eqref{B-0-0} to see
\[ 
\frac{\di}{\di t} \int_{\mathbb{R}^{2d}}e^{\alpha |\bx|} \mu_t(\di \bx,\di \bw) = \alpha \int_{\mathbb{R}^{2d}} e^{\alpha |\bx|}  \Big \langle \frac{\bw}{{\tilde F}}, \frac{\bx}{|\bx|} \Big \rangle  \mu_t(\di\bx,\di \bw) \leq c \alpha  \int_{\mathbb{R}^{2d}} e^{\alpha |\bx|} \mu_t(\di\bx,\di \bw).
\]
This yields the desired estimate. 
\end{proof}
As a direct corollary of Lemma \ref{L2.5}, we have the following estimate.
\begin{corollary} \label{C2.2}
Let $\mu \in L^{\infty}([0,T);\mathcal{P}(\mathbb{R}^{2d}))$ be a measure-valued solution to \eqref{RKCS} with the initial datum $\mu_0$ satisfying 
\[
 \int_{\mathbb{R}^{2d}}  \mu_0(\di \bx,\di \bw) = 1, \quad {\mathcal M}_e(\mu_0, \alpha)  < \infty,
\]
and let $R_x(t)$ and $R_w(t)$ be arbitrary nonnegative locally bounded functions. Then, the following estimates hold:~for $t \geq 0$, 
\begin{align}
	\begin{aligned} \label{B-0-0-1}
		& (i)~ \int_{|\bw_{\star}|\ge R_w(t)} \mu_t(\di \bx_{\star},\di \bw_{\star}) \le \mathcal{M}_{e}(\mu_0, \alpha) e^{ -\alpha R_w(t)}. \\
		& (ii)~ \int_{|\bx_{\star}|\ge R_x(t)} \mu_t(\di \bx_{\star},\di \bw_{\star}) \le \mathcal{M}_{e}(\mu_0, \alpha) e^{\alpha (ct-R_x(t))}.
	\end{aligned}
\end{align}\[  \]
\end{corollary}
\begin{proof} 
	It follows from the first estimate in Lemma \ref{L2.5} that 
	\begin{align*}
		\begin{aligned}
			& e^{\alpha R_w(t)}\int_{|\bw_{\star}|\ge R_w(t)} \mu_t(\di \bx_{\star},\di \bw_{\star}) \\
			& \hspace{1cm} \leq \int_{\mathbb{R}^{2d}}e^{\alpha |\bw|} \mu_t(\di \bx,\di \bw) \le \int_{\mathbb{R}^{2d}}e^{\alpha |\bw|} \mu_0(\di \bx,\di \bw)  \leq  \mathcal{M}_{e}(\mu_0, \alpha).
		\end{aligned}
	\end{align*}
	Next, it follows from the second estimate in Lemma \ref{L2.5} that 
\begin{align*}
\begin{aligned}
& e^{\alpha R_x(t)}\int_{|\bx_{\star}|\ge R_x(t)} \mu_t(\di \bx_{\star},\di \bw_{\star}) \\
& \hspace{1cm} \leq \int_{\mathbb{R}^{2d}}e^{\alpha |\bx|} \mu_t(\di \bx,\di \bw) \le \Big( \int_{\mathbb{R}^{2d}}e^{\alpha |\bx|} \mu_0(\di \bx,\di \bw) \Big) e^{c\alpha t} \leq  \mathcal{M}_{e}(\mu_0, \alpha) e^{c\alpha t}.
\end{aligned}
\end{align*}
This implies the desired estimate.
\end{proof}		 
Now, we set the radius of the projection onto $\bbr_{\bw}^d$ for the support of the initial measure $\mu_0$ as $R_{\bw}(0)$:		
		\[
		R_{\bw}(0):=\left\{\sup |\bw|:~(\bx, \bw) \in {\rm supp} \mu_0 \right\} <\infty.
		\]
\begin{proposition}
\emph{\cite{Ha2021,w7}}
\begin{enumerate}
\item
Suppose that there exist positive constants \(\delta\) and \(D_{\bx}^{\infty}\) such that
		\[
\phi(D_{\bx}^{\infty}) = \sqrt{1 + \frac{ D_{\bw}^2}{c^2}}\left( \frac{4 D_{\bw}^2}{c^2} + \delta\right),  \]
and diameters of position and velocity supports for initial datum satisfy 
		\[
		D_{\bx}(0) < D_{\bx}^{\infty}, \quad (D_{\bx}^{\infty} - D_{\bx}(0))\delta > \left(1 + \frac{2c D_{\bw}^2}{2} \right) D_{\bw}(0),
		\]
and let $\mu\in L^\infty([0,T);\mathcal{P}(\mathbb{R}^{2d}))$ be a measure-valued solution to \eqref{RKCS}. Then, the following assertions hold. \newline
\begin{enumerate}
\item
Strong flocking emerges exponentially fast: there exists a positive constant $\delta>0$ such that
		\[
		\sup_{0 \leq t < \infty} D_{\bx}(t) < D_{\bx}^{\infty}, \quad D_{\bw}(t) \leq D_{\bw}(0)e^{-\delta t}, \quad t \geq 0.\]
\item		
Velocity support is uniformly bounded:
		\[\sup\limits_{0\le t<\infty}\left\{\sup |\bw| \, ~\Big| ~\, \exists~\bx \in \mathbb{R}^d \text{ such that } f(t,\bx, \bw) \neq 0\right\}\le R_{\bw}(0).\]
\end{enumerate}
\vspace{0.1cm}
\item
Let $\mu, \nu \in L^\infty([0,T);\mathcal{P}(\mathbb{R}^{2d}))$ be measure-valued solutions to \eqref{RKCS} corresponding to initial measures $\mu_0,\nu_0\in\mathcal{P}_p(\mathbb{R}^{d})$, respectively. Then, for any $T>0$, there exists some positive constant $C = C(c, \beta, p, T)$ such that
		\begin{equation*}
		\sup_{0 \leq t < T}W_p(\mu_t, \nu_t)\le C W_p(\mu_0,\nu_0).
		\end{equation*}
\end{enumerate}		
	\end{proposition}
	\begin{remark} 
The existence, uniqueness stability results of non-compact solutions to system \eqref{RKCS} can be derived from standard approximation technique and $|\frac{\bw}{{\tilde F}}|\le c.$ More details, we refer to \cite{w7}.
	\end{remark}
For $(\bx, \bw) \in \bbr^{2d}$, we define the forward particle trajectory (or bi-characteristic), simply characteristics $(\bx(t), \bw(t)) = (\bx(t,0,\bx,\bw), \bw(t,0,\bx,\bw))$ as the solution for the following Cauchy problem:
\begin{equation} \label{RKCS-3}
			\begin{cases}
			\displaystyle \dot \bx(t)=\frac{\bw(t)}{{\tilde F}(t)}, \quad t > 0, \\
			\displaystyle \dot \bw(t)={\bf F}_a[\mu_t], \\
			\displaystyle (\bx(0), \bw(0)) = (\bx, \bw).
			\end{cases}
\end{equation}
Since ${\bf F}_a[\mu]$ is continuous in $t$, locally Lipschitz continuous in $(\bx,\bw)$ and sub-linear in $\bx, \bw$, the Cauchy problem \eqref{RKCS-3} admits a globally well-posedness, and $(\bx(t), \bw(t))$ provides globally well-defined homeomorphism for each fixed time $t$. Moreover, it satisfies 		
\begin{equation} \label{RKCS-4}
			\int_{\mathbb{R}^{2d}}\psi(\bx,\bw)\mu_t(\di \bx,\di \bw)=\int_{\mathbb{R}^{2d}}\psi(\bx(t), \bw(t))\mu_0(\di \bx,\di \bw), \quad \forall \ \psi\in C_b^1(\mathbb{R}^{2d}).
\end{equation}	
Before we move on to the brief summary of our main results in the next section, we state a Grönwall-type lemma to be used in a later section.
\begin{lemma} \label{L2.6}
\emph{\cite{bi-cluster2}}
Let $y: \bbr \to [0, \infty)$ be a differentiable function satisfying the following differential inequality:
\[ y^{\prime}(t) \leq - p(t) y(t) + q(t), \quad t \in \bbr, \]
where $p$ and $q$ are non-negative and non-increasing functions. Then, $y$ satisfies 
\[ y(t) \leq y_0 e^{-\int_0^t p(\tau) \di\tau} + e^{-\int_{\frac{t}{2}}^t p(\tau) \di\tau} \int_0^{\frac{t}{2}} q(\tau) \di\tau + q \Big(\frac{t}{2} \Big) \int_{\frac{t}{2}}^t e^{-\int_s^t p(\tau) \di\tau} \di s, \quad t \geq 0.\]
\end{lemma}
\begin{proof}
For a detailed proof, we refer to \cite{bi-cluster2}.
\end{proof}

\section{Discussion of main results} \label{sec:3}
\setcounter{equation}{0}
In this section, we briefly summarize our main results on the quantitative weak flocking estimate of measure-valued solutions to \eqref{RKCS}. For this, we must have some decay information for the measure-valued solution $\mu_t$ in far-field in phase space. To derive weak flocking estimates,  we consider
\[
\mbox{either}~~{\mathcal M}_e(\mu_0, \alpha)<\infty \quad  \mbox{or} \quad {\mathcal M}_p(\mu_0, D) < \infty.
\]
Note that if for some $\alpha > 0$, 
\[ {\mathcal M}_e(\mu_0, \alpha)  < \infty, \]
then we have
\[ {\mathcal M}_p(\mu_0, D) > 0, \quad \mbox{for all $D \in {\mathbb N}$}. \]
Thus, we first focus on the polynomially decaying distributions.

\subsection{Polynomially decaying distributions}\label{sec:3.1} In this part, we derive the weak flocking estimate for the algebraically decaying distribution class:
\begin{equation} \label{N-2}
\int_{\mathbb{R}^{2d}} \Big (|\bx|^D+ |\bw|^D \Big) \mu_0(\mathrm{\di}\bx, \mathrm{\di}\bw) < \infty.
\end{equation}
Under the setting \eqref{N-2}, we derive the weak flocking estimate in three steps:  \newline   
\begin{itemize}
\item
First step: We identify the time-dependent effective domain $\Omega_{\mbox{eff}}$ so that the total mass outside this effective domain decays to zero asymptotically. 
\vspace{0.1cm}
\item
Second step:  We set the velocity alignment functional ${\mathcal L}$:
\[ {\mathcal L} (t):= \int_{\mathbb{R}^{2d}} |\bw - \bw_c|^2 \mu_t(\di \bx, \di  \bw). \]
Then, we estimate the time-derivate of ${\mathcal L}$ on the effective domain and extract the positive lower bound of the kernel function $\phi$ on it, and use the decay estimate of the mass outside the effective domain to derive a differential inequality, 
\begin{equation} \label{N-3}
	\frac{\di {\mathcal L}}{\di t} \leq  -|{\mathcal O}(1)| (1 + t)^{-(\beta \gamma_p + 3 \delta_p)}  {\mathcal L}  +  {\mathcal O}(1) (1 + t)^{-  \delta_p \left(\frac{(l_1-1)D}{l_1} + 3\right)} +  {\mathcal O}(1)  (1+t)^{-\frac{D}{2} (\gamma_p -1)}.
\end{equation}
Furthermore, we use the generalized Grönwall's type lemma (Lemma \ref{L2.6}) to derive the decay estimate of ${\mathcal L}$.
\vspace{0.2cm}
\item
Third step: We use this time decay of ${\mathcal L}$ to derive the uniform bound estimate for the spatial cohesion.  
\end{itemize}

\vspace{0.5cm}

\noindent Now, we are ready to introduce an effective time-varying domain:
\begin{align}
\begin{aligned} \label{N-4}
& R_x(t) :=  R^0_x(1+ t)^{\gamma_p},  \quad \mbox{for some}~~\gamma_p > 1, \\
& R_{w}(t) :=   R_{w}^0  (1 + t)^{\delta_{p}}, \quad \mbox{for some $\delta_{p} > 0$}, \\
& \Omega_{\mathrm{eff}}(t) := \{ (\bx, \bw) \in \bbr^{2d}:~|\bx| \leq R_x(t) \quad |\bw| \leq R_w(t) \} . 
\end{aligned}
\end{align}
In the sequel, we will use a positive constant $C$ which depends on parameters, and this constant will vary from context to context. In the following lemma, we provide several estimates to be used in the later part of this paper. We set
\begin{equation} \label{N-4-1}
\underline{\phi} (t) := \inf\limits_{\bx(t),\bx_{\star}(t)\in B_{R_x(t)}} \phi(|\bx(t)-\bx_{\star}(t)|). 
\end{equation}
Our first main result can be stated as follows.
\begin{theorem}\label{T3.1}
For $D \in {\mathbb N}$ and $T \in (0, \infty]$, let $\mu \in L^{\infty}([0,T);\mathcal{P}_D(\mathbb{R}^{2d}))$ be a measure-valued solution to \eqref{RKCS} with initial datum $\mu_0$ satisfying
\[  \int_{\mathbb{R}^{2d}}  \mu_0(\di \bx,\di \bw) = 1 \quad \mbox{and} \quad {\mathcal M}_p(\mu_0, D)  < \infty. \]
 Then, the following assertions hold:
\begin{enumerate}
\item
(Velocity alignment):
If nonnegative constants $\beta, \gamma_p, \delta_p, D, l_1$ satisfy
\[ l_1>1,\quad D \geq 2l_1,\quad 0 \leq \beta < 1, \quad \gamma_p > 1, \quad 0 \leq \beta \gamma_p <1, \quad  0 <  \delta_p  < \frac{1-\beta \gamma_p}{3}, \]
then weak mono-cluster flocking \eqref{A-4} emerges algebraically fast:
\[  \int_{\mathbb{R}^{2d}} |\bw - \bw_c|^2 \mu_t(\di \bx, \di  \bw) \lesssim (1 + t)^{-\min \left\{  \delta_p \left(\frac{(l_1-1)D}{l_1} + 3\right),~\frac{D}{2} (\gamma_p -1)  \right\}}, \quad t \gg 1.   \]
\item
(Spatial cohesion):
If $D \in {\mathbb N}$ is sufficiently large such that 
\[ 
 \quad l_1>1, \quad D \geq 2l_1, \quad   \min \left\{  \delta_p \left(\frac{(l_1-1)D}{l_1} + 3\right),~\frac{D}{2} (\gamma_p -1)  \right\} > 2, 
\]
then spatial cohesion emerges asymptotically:
\[
\sup_{0 \leq t < \infty} \int_{\mathbb{R}^{2d}}  \Big |\bx-\frac{\bw_c(0)}{{\tilde F}}t-\bx_c(0) \Big |^2\mu_t(\di \bx,\di \bw) < \infty.
\]
\end{enumerate}
\end{theorem}
\begin{proof}
Since the proofs are very lengthy, we leave them in Section \ref{sec:4}.
\end{proof}
%
\subsection{Exponentially decaying distributions}\label{sec:3.2} 
In this part, we derive the weak flocking estimate for the exponentially decaying distribution class: 
\begin{equation} \label{N-5}
{\mathcal M}_e(\mu_0, \alpha) =  \int_{\mathbb{R}^{2d}} e^{\alpha\left(|\bx|+|\bw|\right)} \mu_0(\mathrm{\di}\bx, \mathrm{\di }\bw)  < \infty,
\end{equation}
for $\alpha > 0$. Since 
\[ 
{\mathcal M}_e(\mu_0, \alpha) < \infty \quad \mbox{for some $\alpha > 0$} \quad \Longrightarrow \quad  {\mathcal M}_p(\mu_0, D) < \infty \quad \mbox{for all $D \geq 0$},
\]
the weak mono-cluster flocking in Theorem \ref{T3.1} can be applied for exponentially decaying distributions as well. Therefore, the analysis is almost the same as the polynomially decaying distributions except for variations of constants, although some conditions on parameters can be relaxed. 
\begin{align}
\begin{aligned} \label{N-6}
& R_x(t) :=  R^0_x +  \gamma_e t \quad \mbox{for some}~~\gamma_e > c, \\
& R_{w}(t) :=   R_{w}^0(1 + t)^{\delta_e}, \quad \mbox{for some $\delta_e > 0$}, \\
&  \Omega_{\mathrm{eff}}(t) := \{ (\bx, \bw) \in \bbr^{2d}:~|\bx| \leq R_x(t) \quad |\bw| \leq R_w(t) \}.
\end{aligned}
\end{align}
Note that compared to \eqref{N-4} in which $R_x(t)$ has a super-linear growth, $R_x(t)$ defined in \eqref{N-6} has a linear growth. It follows from Corollary \ref{C2.2} that 
\begin{align}
\begin{aligned} \label{N-7}
& \int_{|x_{\star}|\ge R_x(t)} \mu_t(\di \bx_{\star},\di \bw_{\star}) \\
& \hspace{1cm} \le \mathcal{M}_{e}(\mu_0, \alpha) e^{\alpha (ct-R_x(t))} = \mathcal{M}_{e}(\mu_0, \alpha) e^{-\alpha R_x^0} e^{-\alpha(\gamma_e -c)t} \to 0, \quad \mbox{as $t \to \infty$}. 
\end{aligned}
\end{align}
\begin{theorem}\label{T3.2}
Suppose that $\gamma_e, \delta_e, \beta$ and $\alpha$ satisfy
\[
\alpha > 0, \quad 0 \leq \beta < 1, \quad \gamma_e > c, \quad 0 <  \delta_e < \frac{1-\beta}{3},
\]
and let $\mu \in L^{\infty}([0,T);\mathcal{P}(\mathbb{R}^{2d}))$ be a measure-valued solution to \eqref{RKCS} with the initial datum $\mu_0$ satisfying
\[  \int_{\mathbb{R}^{2d}}  \mu_0(\di \bx,\di \bw) = 1 \quad \mbox{and} \quad  {\mathcal M}_e(\mu_0, \alpha) < \infty. \]
 Then, the following assertions hold. 
 \begin{enumerate}
 \item
 Velocity alignment occurs exponentially fast:
 \[  \int_{\mathbb{R}^{2d}} |\bw - \bw_c|^2 \mu_t(\di \bx, \di  \bw) \lesssim{\tilde C}_{10} \left(e^{-{\tilde C}_9 t^{1-(3\delta_e + \beta)}}+e^{-\frac{\alpha}{4}(1+t)^{\delta_e}} \right). \]
\item Spatial cohesion occurs uniformly in time:
\[
\sup_{0 \leq t < \infty} \int_{\mathbb{R}^{2d}}  \Big |\bx-\frac{\bw_c(0)}{{\tilde F}}t-\bx_c(0) \Big |^2\mu_t(\di \bx,\di \bw) < \infty.
\]
\end{enumerate}
\end{theorem}
\begin{proof}
Since the proof is very lengthy, we leave its proof in Section \ref{sec:5}.
\end{proof}
\begin{remark}
	We provide some comments on Theorem \ref{T3.1}, Theorem \ref{T3.2} and results in the previous papers.
	\begin{enumerate}
		
		\item 	First, compared to the previous series of compact support work \cite{Ha2020,Ha2021}, we extend their compact setting to a fully non-compact setting. If the solution is compactly supported, our results can be seen as flocking dynamics also. \vspace{0.1cm}
		
		\item 	Second, compared to the classical KCS model \cite{kinetic4,w2,w6}, the velocity coupling term of the RKCS model is nonlinear, which makes weak flocking analysis difficult.\vspace{0.1cm}
		
		\item 	Third, compared to the classical KCS model with a fully non-compact setting \cite{w8}, we can get the exponential convergence of weak flocking behavior, even when the velocity is not compactly supported. In contrast, exponential convergence of the classical KCS model can be established only when the velocity variable is compactly supported.\vspace{0.1cm}
		
		\item Finally, classical limits have been studied in previous papers under compact support\cite{rcs1,rcs3,HLR}, but whether we can establish the weak flocking behavior through classical limit with non-compact support is not known yet.
	\end{enumerate}
\end{remark}
\section{Emergence of weak flocking I}\label{sec:4}
\setcounter{equation}{0}
In this section, we study the quantitative weak flocking dynamics of measure-valued solutions to \eqref{RKCS} in the class of polynomially decaying distribution class:
\[  \int_{\mathbb{R}^{2d}} \Big (|\bx|^D+ |\bw|^D \Big) \mu_t(\mathrm{\di}\bx, \mathrm{\di}\bw) < \infty. \]
\subsection{Velocity alignment}\label{sec:4.1} In this subsection, we derive quantitative velocity alignment for the algebraically decaying distribution class. Suppose that the initial measure $\mu_0$ satisfies 
\begin{equation} \label{C-0-1}
{\mathcal M}_p(\mu_0, D) = \int_{\mathbb{R}^{2d}} \Big (|\bx|^D+ |\bw|^D \Big) \mu_0(\mathrm{\di}\bx, \mathrm{\di}\bw) < \infty. 
\end{equation}
Let $(\bx, \bw), (\bx_{\star}, \bw_{\star})$ be two points in phase space $\bbr^{2d}$ and we consider the particle trajectories  issued from $(\bx, \bw), (\bx_{\star}, \bw_{\star})$, respectively:
\begin{align*}
\begin{aligned}
& (\bx(t), \bw(t)) := (\bx(t,0,\bx, \bw), \bw(t,0,\bx, \bw)), \\
& (\bx_{\star}(t), \bw_{\star}(t)) := (\bx(t,0,\bx_{\star}, \bw_{\star}), \bw(t,0,\bx_{\star}, \bw_{\star})),
\end{aligned}
\end{align*}
i.e., they are solutions to the following Cauchy problems for ODEs:
\[
	\begin{cases}
	\displaystyle \dot \bx(t)=\frac{\bw(t)}{{\tilde F}(t)}, \quad t > 0,\vspace{0.2cm} \\
	\displaystyle  \dot \bw(t)={\bf F}_a[\mu_t](t,\bx(t), \bw(t)), \vspace{0.2cm}\\
	\displaystyle (\bx(0), \bw(0)) = (\bx, \bw),
	\end{cases}
\quad \mbox{and} \qquad 
\begin{cases}
	\displaystyle \dot \bx_{\star}(t)=\frac{\bw_{\star}(t)}{{\tilde F}_{\star}(t)},  \quad t > 0,\vspace{0.2cm}\\
	\displaystyle  \dot \bw_{\star}(t)={\bf F}_a[\mu_t](t,\bx_{\star}(t), \bw_{\star}(t)), \vspace{0.2cm}\\
	\displaystyle (\bx_\star(0), \bw_\star(0)) = (\bx_\star, \bw_\star).
	\end{cases}
\]
In the following lemma, we provide several estimates to be used in the later part of this paper. 
\begin{lemma} \label{L4.1}
For $D \in {\mathbb N}$ and $T \in (0, \infty]$, let $\mu \in L^{\infty}([0,T);\mathcal{P}(\mathbb{R}^{2d}))$ be a measure-valued solution to \eqref{RKCS} with initial datum $\mu_0$ satisfying
\[  \int_{\mathbb{R}^{2d}}  \mu_0(\di \bx,\di \bw) = 1, \quad {\mathcal M}_p(\mu_0, D)  < \infty. \]
Then, there exist positive constants $C_1 = C_1( D, {\mathcal M}_p(\mu_0, D), R_x^0, c^D)$ and $C_2 = C_2(\beta, R_x^0)$  such that 
\[
\int_{|\bx| \geq R_x(t)} \mu_t(\di \bx, \di \bw) \leq C_1(1 + t)^{-D(\gamma_p - 1)}, \quad C_2 (1+t)^{-\beta \gamma_p} \leq  \underline{\phi} (t) \leq 1 \quad \mbox{for}~ t \gg 1,
\]
where $\gamma_p$ and $\underline{\phi}$ are defined in \eqref{N-4} ad \eqref{N-4-1}, respectively. 
\end{lemma}
\begin{proof}
\noindent (i)~ It follows from Corollary \ref{C2.1} and Lemma \ref{L2.4} that 
\begin{align}
\begin{aligned} \label{C-1}
 &\int_{|\bx| \geq R_x(t)} \mu_t(\di \bx, \di \bw)  \le \frac{   2^{D-1} \max \Big \{  {\mathcal M}_p(\mu_0, D), c^D  \Big\} (1 + t)^D }{R_x(t)^D} \\
 & \hspace{1cm} \leq \frac{   2^{D-1} \max \Big \{  {\mathcal M}_p(\mu_0, D), c^D \Big\} }{(R^0_x)^D} (1 + t)^{-D(\gamma_p - 1)}.
 \end{aligned}
 \end{align}
 \vspace{0.2cm}
 
 \noindent (ii)~For a given $t > 0$, we consider two particle trajectories inside the effective domain:
\[ |\bx(t)|\le R_x(t), \quad |\bx_{\star}(t)| \leq R_x(t). \]
Then, we have
\begin{align*}
\begin{aligned}
& | \bx(t)-\bx_{\star}(t)|\le 2 R_x(t) \quad \mbox{and} \\
&  \underline{\phi} (t) :=  \inf\limits_{\bx(t),\bx_{\star}(t) \in B_{R_x(t)}} \phi(|\bx(t)-\bx_{\star}(t)|) \\
& \hspace{1cm} \ge \phi(2R_x(t))  = \frac{1}{\left(1+ 4 |R^0_x|^2 (1 + t)^{2\gamma_p}\right)^{\frac{\beta}{2}}}  \gtrsim \frac{1}{2^{\frac{3\beta}{2}} |R_x^0|^{\beta}} (1+t)^{-\beta \gamma_p},
\end{aligned}
\end{align*}
where we use the relation:
\[ \left(1+ 4 |R^0_x|^2 (1 + t)^{2\gamma_p}\right)^{\frac{\beta}{2}} \lesssim 2^{\frac{3\beta}{2}} |R^0_x|^{\beta} (1 + t)^{ \beta \gamma_p}, \quad t \gg 1. \]
\end{proof}

\begin{lemma} \label{L4.2}
For $D \in {\mathbb N}$, let $\mu \in L^{\infty}([0,T);\mathcal{P}(\mathbb{R}^{2d}))$ be a measure-valued solution to \eqref{RKCS} with initial datum $\mu_0$ satisfying
\[  \int_{\mathbb{R}^{2d}}  \mu_0(\di \bx,\di \bw) = 1, \quad {\mathcal M}_p(\mu_0, D)< \infty,  \]
and let $\bw(t)$ and $\bw_{\star}(t)$ be particle trajectories such that 
\[  |\bw(t) |\le R_{w}(t), \quad  |\bw_{\star}(t)| \le R_{w}(t). \]
Then, there exists a positive constant $C_3 =  C_3 (R_w^0)$ such that 
\[ \left\langle \bw_{\star}(t)- \bw(t),\frac{\bw_{\star}(t)}{{\tilde F}_{\star}} - \frac{\bw(t)}{{\tilde F}} \right\rangle  \geq C_3 (1 + t)^{-3\delta_p} |\bw_{\star}(t)- \bw(t)|^2. \]
\end{lemma}
\begin{proof}
Consider the particle trajectories $\bw(t)$ and $\bw_{\star}(t)$ confined in $R_w$-ball:
\[ |\bw(t) |\le R_{w}(t) \quad \mbox{and} \quad   |\bw_{\star}(t)| \le R_{w}(t).\]
Then, it follows from Lemma \ref{L2.5} that there exists a time-dependent function $\Lambda(t)$ such that
\begin{equation} \label{C-4}
 \left\langle \bw_{\star}(t)- \bw(t),\frac{\bw_{\star}(t)}{{\tilde F}_{\star}} - \frac{\bw(t)}{{\tilde F}} \right\rangle  \geq \Lambda(t) |\bw_{\star}(t)- \bw(t)|^2,
\end{equation}
where we used the hand notation ${\tilde F}_{\star} \equiv {\tilde F}(\bw_{\star})$.  In fact, $\Lambda$ satisfies the following decay estimate: 
\begin{equation} \label{C-5}
\Lambda(t) := \mathcal{O}\left(\frac{1}{2\left(1+\frac{|R_{w}(t)|^2}{c^2}\right)^{\frac{3}{2}}} \right)\gtrsim \frac{1}{2\left(1+ |R^0_{w}|^2 (1 + t)^{2\delta_p}  \right)^{\frac{3}{2}}} \gtrsim \frac{1}{4 \sqrt{2} \max\{1,  |R^0_{w}|^3 \} }  (1 + t)^{-3\delta_p},
\end{equation}
where we used the following relation:
\[ 2\left(1+ |R^0_{w}|^2 (1 + t)^{2\delta_p}  \right)^{\frac{3}{2}} \lesssim  2 \sqrt{2} (1 + |R^0_{w}|^3 (1 + t)^{3\delta_p}) \lesssim 4 \sqrt{2} \max\{1,  |R^0_{w}|^3 \}  (1 + t)^{3\delta_p}. \]
\end{proof}
\subsubsection{Derivation of differential inequality for ${\mathcal L}$} \label{sec:4.1.1}
Note that  Lemma \ref{L2.2}  and the conservation law $\bw_c(t) = \bw_c(0)$ imply
\begin{equation} \label{C-6}
 \dfrac{\di}{\di t} {\mathcal L}(t) =-\int_{\bbr^{4d}} \left\langle \bw_{\star}-\bw,\frac{\bw_{\star}}{{\tilde F}_{\star}} - \frac{\bw}{{\tilde F}} \right\rangle   \phi(|\bx-\bx_\star|) \mu_t(\di \bx,\di \bw) \mu_t(\di \bx_{\star},\di \bw_{\star})\le 0.
\end{equation}
In what follows, we estimate the right-hand side of \eqref{C-6} as
\[ - |{\mathcal O}(1)| (1 + t)^{-(3\delta_p + \beta \gamma_p)} {\mathcal L} + \mbox{time-decaying terms}. \]
Then, it follows from \eqref{RKCS-4} that the relation \eqref{C-1} can be rewritten and estimated using particle trajectories:
\begin{align}
\begin{aligned} \label{C-7}
&\frac{\di}{\di t} {\mathcal L}(t) = \dfrac{\di}{\di t}\int_{\mathbb{R}^{2d}}|\bw(t)-\bw_c|^2\mu_0(\di \bx,\di \bw)\\ 
&=-\int_{\mathbb{R}^{4d}} \phi(|\bx(t)-\bx_{\star}(t)|) \left\langle \bw_{\star}(t)-\bw(t),\frac{\bw_{\star}(t)}{{\tilde F}_{\star}} - \frac{\bw(t)}{{\tilde F}} \right\rangle \mu_0(\di \bx,\di \bw) \mu_0(\di \bx_{\star},\di\bw_{\star})\\
& \leq -\underline{\phi}(t) \int_{|\bx(t)|\le R_x(t)}\int_{|\bx_{\star}(t)|\le R_x(t)} \left\langle \bw_{\star}(t)-\bw(t),\frac{\bw_{\star}(t)}{{\tilde F}_{\star}} - \frac{\bw(t)}{{\tilde F}} \right\rangle \mu_0(\di \bx,\di \bw) \mu_0(\di \bx_{\star},\di\bw_{\star})\\
& = -\underline{\phi}(t) \int_{\mathbb{R}^{4d}}\left\langle \bw_{\star}(t)- \bw(t),\frac{\bw_{\star}(t)}{{\tilde F}_{\star}} - \frac{\bw(t)}{{\tilde F}} \right\rangle \mu_0(\di \bx,\di \bw) \mu_0(\di \bx_{\star},\di\bw_{\star})\\
&+ \underline{\phi}(t)  \int_{\mathbb{R}^{2d}} \int_{|\bx_{\star}(t)|> R_x(t)}\left\langle \bw_{\star}(t)-\bw(t),\frac{\bw_{\star}(t)}{{\tilde F}_{\star}} - \frac{\bw(t)}{{\tilde F}} \right\rangle \mu_0(\di \bx,\di \bw) \mu_0(\di \bx_{\star},\di\bw_{\star})\\
&+\underline{\phi}(t) \int_{|\bx(t)|> R_x(t)}\int_{\mathbb{R}^{2d}}\left\langle \bw_{\star}(t)-\bw(t),\frac{\bw_{\star}(t)}{{\tilde F}_{\star}} - \frac{\bw(t)}{{\tilde F}} \right\rangle \mu_0(\di \bx,\di \bw) \mu_0(\di \bx_{\star},\di\bw_{\star})\\
&=:{\mathcal I}_{11} + {\mathcal I}_{12} + {\mathcal I}_{13},
\end{aligned}
\end{align}
where we use point-wise non-negativity:
\[
\left\langle \bw_{\star}-\bw,\frac{\bw_{\star}}{{\tilde F}_{\star}} - \frac{\bw}{{\tilde F}} \right\rangle \geq 0.
\]
In the following two lemmas, we estimate the term ${\mathcal I}_{1i},~i=1,2,3$ one by one. 
\begin{lemma}  \label{L4.3} 
There exists a positive constant $C_4 = C_4 (c, \mathcal{M}_p(\mu_0,2), C_1)$ such that 
\[ |{\mathcal I}_{12}(t)| \leq \frac{C_4}{2}  (1+t)^{-\frac{D}{2} (\gamma_p -1)}, \quad  |{\mathcal I}_{13}(t)| \leq  \frac{C_4}{2}  (1+t)^{-\frac{D}{2} (\gamma_p -1)}. \]
\end{lemma}
\begin{proof}
By symmetry, we have ${\mathcal I}_{12} = {\mathcal I}_{13}$. Hence, it suffices to derive a decay estimate for ${\mathcal I}_{12}$ only. To derive an polynomial decay of ${\mathcal I}_{12}$, we use the Cauchy-Schwarz inequality and simple estimates
\[  \Big |\frac{\bw}{{\tilde F}}  \Big |\le c \quad \mbox{and} \quad \Big |\frac{\bw_{\star}}{{\tilde F}_{\star}} \Big |\le c \]
 to see
\begin{align}
	\begin{aligned} \label{C-8}
		{\mathcal I}_{12} &=  \underline{\phi}(t)  \int_{\mathbb{R}^{2d}} \int_{|\bx_{\star}(t)|> R_x(t)}\left\langle \bw_{\star}(t)-\bw(t),\frac{\bw_{\star}(t)}{{\tilde F}_{\star}} - \frac{\bw(t)}{{\tilde F}} \right\rangle \mu_0(\di \bx,\di \bw) \mu_0(\di \bx_{\star},\di\bw_{\star})\\
		&\le 2c \underline{\phi}(t)  \int_{\mathbb{R}^{2d}}\int_{|\bx_{\star}(t)|> R_x(t)} \Big( | \bw_{\star}(t)|+|\bw(t)| \Big) \mu_0(\di \bx,\di \bw) \mu_0(\di \bx_{\star},\di\bw_{\star})\\
		&\le2c \underline{\phi}(t)  \left(\int_{\mathbb{R}^{2d}}|\bw(t)|^2 \mu_0(\di \bx,\di \bw)\right)^{\frac{1}{2}}\int_{|\bx_{\star}(t)|> R_x(t)}\mu_0(\di \bx_{\star},\di\bw_{\star})\\
		&+2c \underline{\phi}(t)  \left(\int_{\mathbb{R}^{2d}}|\bw(t)|^2 \mu_0(\di \bx,\di \bw)\right)^{\frac{1}{2}}\left(\int_{|\bx_{\star}(t)|> R_x(t)}\mu_0(\di \bx_{\star},\di\bw_{\star})\right)^{\frac{1}{2}}\\
	&	\le 2c \underline{\phi}(t)  \left(\mathcal{M}_p(\mu_0,2)\right)^{\frac{1}{2}}\left[ \left(\int_{|\bx_{\star}(t)|> R_x(t)}\mu_0(\di \bx_{\star},\di\bw_{\star})\right)^{\frac{1}{2}}+\int_{|\bx_{\star}(t)|> R_x(t)}\mu_0(\di \bx_{\star},\di\bw_{\star})\right ].
	\end{aligned}
\end{align}
Here,  in the third inequality, we used the Cauchy-Schwarz inequality to find 
\begin{align*}
\begin{aligned}
&\int_{\mathbb{R}^{2d}}\int_{|\bx_{\star}(t) |> R_x(t)}| \bw_{\star}(t)|\mu_0(\di \bx,\di \bw) \mu_0(\di \bx_{\star},\di\bw_{\star})\\
& \hspace{0.5cm} =\int_{|\bx_{\star}(t)|> R_x(t)}| \bw_{\star}(t)| \mu_0(\di \bx_{\star},\di\bw_{\star}) \cdot \underbrace{\int_{\mathbb{R}^{2d}}  \mu_0(\di \bx,\di \bw)}_{= 1} \\
& \hspace{0.5cm} \le\left(\int_{\mathbb{R}^{2d}}|\bw(t)|^2 \mu_0(\di \bx,\di \bw)\right)^{\frac{1}{2}}\left(\int_{|\bx_{\star}(t)|> R_x(t)}\mu_0(\di \bx_{\star},\di\bw_{\star})\right)^{\frac{1}{2}}
\end{aligned}
\end{align*}
and 
\begin{align*}
\begin{aligned}
	&\int_{\mathbb{R}^{2d}}\int_{|\bx_{\star}(t)|> R_x(t)}| \bw(t)|\mu_0(\di \bx,\di \bw) \mu_0(\di \bx_{\star},\di\bw_{\star})\\
	& \hspace{0.5cm} =\int_{|\bx_{\star}(t)|> R_x(t)}\mu_0(\di \bx_{\star},\di\bw_{\star})\int_{\mathbb{R}^{2d}}|\bw(t)| \mu_0(\di \bx,\di \bw)\\
	& \hspace{0.5cm} \le\left(\int_{\mathbb{R}^{2d}}|\bw(t)|^2 \mu_0(\di \bx,\di \bw)\right)^{\frac{1}{2}}\int_{|\bx_{\star}(t)|> R_x(t)}\mu_0(\di \bx_{\star},\di\bw_{\star}).
\end{aligned}
\end{align*}

\vspace{0.2cm}

\noindent On the other hand, it follows from Lemma \ref{L2.2} that 
\begin{equation} \label{C-9}
\int_{|\bx| \geq R_{\bx}(t)} \mu_t(\di \bx, \di \bw)  \leq C_1(1 + t)^{-D(\gamma_p - 1)}. 
\end{equation}
We use \eqref{C-8}, \eqref{C-9} and $\underline{\phi}(t) \leq 1$ to get the desired estimate:
\begin{align*}
\begin{aligned}
{\mathcal I}_{12}(t) &\le 2c \underline{\phi}(t)  \left(\mathcal{M}_p(\mu_0,2)\right)^{\frac{1}{2}} \left[ \left(\int_{|\bx_{\star}(t)|> R_x(t)}\mu_0(\di \bx_{\star},\di\bw_{\star})\right)^{\frac{1}{2}}+\int_{|\bx_{\star}(t)|> R_x(t)}\mu_0(\di \bx_{\star},\di\bw_{\star})\right ] \\
& \leq  2c \left(\mathcal{M}_p(\mu_0,2)\right)^{\frac{1}{2}}  \Big( \sqrt{C_1}(1 + t)^{-\frac{D}{2} (\gamma_p - 1)}   + C_1(1 + t)^{-D(\gamma_p - 1)}  \Big) \\
&\le  4c \left(\mathcal{M}_p(\mu_0,2)\right)^{\frac{1}{2}} \max\{ \sqrt{C_1}, C_1 \}  (1 + t)^{-\frac{1}{2} D(\gamma_p -1)}  \\
&=: \frac{C_4}{2} (1 + t)^{-\frac{1}{2} D(\gamma_p -1)}.
\end{aligned}
\end{align*}
The same analysis can be done for ${\mathcal I}_{13}$ as well. 
\end{proof}
\vspace{0.2cm}

\noindent Next, we return to the estimate of the term ${\mathcal I}_{11}$. It follows from Lemma \ref{L4.1} and Lemma \ref{L4.2} that
\begin{align} 
\begin{aligned} \label{C-10}
{\mathcal I}_{11} &= -\underline{\phi}(t) \int_{\mathbb{R}^{4d}}  \left\langle \bw_{\star}(t)- \bw(t),\frac{\bw_{\star}(t)}{{\tilde F}_{\star}} - \frac{\bw(t)}{{\tilde F}} \right\rangle \mu_0(\di \bx,\di \bw) \mu_0(\di \bx_{\star},\di\bw_{\star})\\	
& \leq  -\underline{\phi}(t)  \int_{ |\bw(t)|\le R_{w}(t)}\int_{|\bw_{\star}(t)|\le R_{w}(t)} \left\langle \bw_{\star}(t)- \bw(t),\frac{\bw_{\star}(t)}{{\tilde F}_{\star}} - \frac{\bw(t)}{{\tilde F}} \right\rangle \mu_0(\di \bx,\di \bw) \mu_0(\di \bx_{\star},\di\bw_{\star})\\	
& \le - C_3 \underline{\phi}(t) (1 + t)^{-3\delta_p} \int_{|\bw(t)|\le R_{w}(t)}\int_{|\bw_{\star}(t)|\le R_{w}(t)} | \bw_{\star}(t)-\bw(t)|^2\mu_0(\di \bx,\di \bw) \mu_0(\di \bx_{\star},\di\bw_{\star}) \\
&=-C_3 \underline{\phi}(t)  (1 + t)^{-3\delta_p} \int_{\mathbb{R}^{4d}} |\bw_{\star}(t)-\bw(t)|^2\mu_0(\di \bx,\di \bw) \mu_0(\di \bx_{\star},\di\bw_{\star})\\
&+ C_3 \underline{\phi}(t)  (1 + t)^{-3\delta_p} \int_{|\bw(t)|> R_{w}(t)}\int_{\mathbb{R}^{2d}} |\bw_{\star}(t)-\bw(t)|^2\mu_0(\di \bx,\di \bw) \mu_0(\di \bx_{\star},\di\bw_{\star})\\
&+ C_3 \underline{\phi}(t)  (1 + t)^{-3\delta_p} \int_{|\bw_{\star}(t)|> R_{w}(t)}\int_{\mathbb{R}^{2d}} |\bw_{\star}(t)-\bw(t)|^2\mu_0(\di \bx,\di \bw) \mu_0(\di \bx_{\star},\di\bw_{\star})\\
&=: {\mathcal I}_{111}+ {\mathcal I}_{112}+ {\mathcal I}_{113}.
\end{aligned}
\end{align}
In the following lemma, we estimate the term ${\mathcal I}_{11i},~i=1,2,3$ one by one. 
\begin{lemma} \label{L4.4}
There exist positive constants $C_5$ and $C_6$ such that 
\begin{align*}
\begin{aligned}
& (i)~{\mathcal I}_{111} =  -C_5(1 + t)^{-(\beta \gamma_p + 3 \delta_p)} {\mathcal L}. \\
& (ii)~{\mathcal I}_{112}  \leq  \frac{C_6}{2} (1 + t)^{- \delta_p \left(\frac{(l_1-1)D}{l_1}+ 3\right)}. \\
& (iii)~{\mathcal I}_{113} \leq  \frac{C_6}{2} (1 + t)^{-  \delta_p \left(\frac{(l_1-1)D}{l_1}+ 3\right)}. \\
& (iv)~{\mathcal I}_{11} \leq  -C_5(1 + t)^{-(\beta \gamma_p + 3 \delta_p)}  {\mathcal L}  +   C_6 (1 + t)^{-  \delta_p \left(\frac{(l_1-1)D}{l_1}+ 3\right)}.
\end{aligned}
\end{align*}
\end{lemma}
\begin{proof} (i) It follows from \eqref{RKCS-4} and Lemma \ref{L2.2} that 
\[
 \int_{\bbr^{2d}}\bw(t) \mu_0(\di \bx,\di \bw) = \int_{\bbr^{2d}}\bw \mu_t(\di \bx,\di \bw)  = \int_{\bbr^{2d}}\bw \mu_0(\di \bx,\di \bw)= \bw_c(0).
\]
On the other hand, since $\bw_c(t) =  \int_{\bbr^{2d}}\bw \mu_t(\di \bx,\di \bw)$, we have
\[
\int_{\bbr^{2d}}\bw(t) \mu_0(\di \bx,\di \bw)  = \bw_c(t), \quad \forall~t > 0. 
\]
This yields
\begin{equation} \label{C-11}
 \int_{\bbr^{2d}} (\bw(t) - \bw_c(t)) \mu_0(\di \bx,\di \bw)  = 0.
\end{equation}
Now, we use  \eqref{C-10} and \eqref{C-11} to see 
\begin{align*}
\begin{aligned} 
{\mathcal I}_{111} &=- C_3 \underline{\phi}(t) (1 + t)^{-3\delta_p} \int_{\mathbb{R}^{4d}} |\bw_{\star}(t)- \bw(t)|^2\mu_0(\di \bx,\di \bw) \mu_0(\di \bx_{\star},\di\bw_{\star})\\
&=-C_3 \underline{\phi}(t) (1 + t)^{-3\delta_p} \int_{\mathbb{R}^{4d}} |\bw_{\star}(t)-\bw_c(t)+\bw_c(t)-\bw(t) |^2\mu_0(\di \bx,\di \bw) \mu_0(\di \bx_{\star},\di\bw_{\star})\\
&=-2  C_3 \underline{\phi}(t) (1 + t)^{-3\delta_p} \int_{\mathbb{R}^{2d}} |\bw(t)-\bw_c(t)|^2\mu_0(\di \bx,\di \bw) \\
& = -2  C_3 \underline{\phi}(t) (1 + t)^{-3\delta_p} {\mathcal L} \\
&\leq -2  C_2 C_3 (1 + t)^{-(\beta \gamma_p + 3 \delta_p)} {\mathcal L} \\
&=: -C_5(1 + t)^{-(\beta \gamma_p + 3 \delta_p)} {\mathcal L}.
\end{aligned}
\end{align*}
\noindent (ii)~By symmetry, it is easy to see that 
\[ {\mathcal I}_{112} = {\mathcal I}_{113}. \]
Hence, it suffices to estimate the term ${\mathcal I}_{112}$. We use the inequality 
\[ |a - b|^2 \leq 2 (|a|^2 + |b|^2) \quad \mbox{and} \quad \underline{\phi} \leq 1 \]
to obtain
	\begin{align} 
		\begin{aligned} \label{C-12}
			{\mathcal I}_{112}&=  C_3 \underline{\phi}(t)  (1 + t)^{-3\delta_p}  \int_{|\bw|> R_w(t)}\int_{\mathbb{R}^{2d}} |\bw_{\star}(t)-\bw(t)|^2\mu_0(\di \bx,\di \bw) \mu_0(\di \bx_{\star},\di\bw_{\star}) \\
			&	\le 2 C_3(1 + t)^{-3\delta_p}   \int_{|\bw(t)|> R_w(t)}\int_{\mathbb{R}^{2d}} \Big(|\bw_{\star}(t)|^2+ |\bw(t)|^2 \Big) \mu_0(\di \bx,\di \bw) \mu_0(\di \bx_{\star},\di\bw_{\star}).
	\end{aligned}
	\end{align}	
Below, we estimate the terms in the right-hand side of \eqref{C-12} one by one. \newline

\noindent $\bullet$~Case A (Estimate of the first term):~We use the Cauchy-Schwarz inequality to derive 
\begin{align}
\begin{aligned} \label{C-13}
	&\int_{|\bw(t)|> R_w(t)}\int_{\mathbb{R}^{2d}} | \bw_{\star}(t)|^2\mu_0(\di \bx,\di \bw) \mu_0(\di \bx_{\star},\di\bw_{\star})\\
	& \hspace{0.5cm} = \Big( \int_{\bbr^{2d}}| \bw_{\star}(t)|^2 \mu_0(\di \bx_{\star},\di\bw_{\star}) \Big) \cdot \Big( \int_{|\bw(t)|> R_w(t)}  \mu_0(\di \bx,\di\bw) \Big)  \\
	& \hspace{0.5cm} \leq    {\mathcal M}_{p}(\mu_0, 2) \mathcal{M}_{p}(\mu_0,D){ (1 + t)^{-D\delta_p}},
\end{aligned}
\end{align}
where we used 
\[
\int_{\bbr^{2d}} | \bw_{\star}(t)|^2 \mu_0(\di \bx_{\star},\di\bw_{\star}) =\int_{\bbr^{2d}} | \bw_{\star}|^2 \mu_t(\di \bx_{\star},\di\bw_{\star}) \leq {\mathcal M}_{p}(\mu_0, 2).
\]
\noindent $\bullet$~Case B (Estimate of the second term): We use Lemma \ref{L2.4} and Corollary \ref{C2.1} to get 
\begin{align}
\begin{aligned} \label{C-144}
	&\int_{|\bw(t)|> R_w(t)}\int_{\mathbb{R}^{2d}} | \bw(t)|^2\mu_0(\di \bx,\di \bw) \mu_0(\di \bx_{\star},\di\bw_{\star})\\
	& \hspace{0.5cm} =\int_{|\bw(t)|> R_w(t)} |\bw(t)|^2 \mu_0(\di \bx,\di\bw)\int_{\mathbb{R}^{2d}}\mu_0(\di \bx_{\star},\di \bw_{\star})\\
	& \hspace{0.5cm} =\int_{|\bw(t)|> R_w(t)} |\bw(t)|^2 \mu_0(\di \bx,\di\bw) \\
	& \hspace{0.5cm} = \int_{|\bw|> R_w(t)} |\bw|^2 \mu_t(\di \bx,\di\bw) \\
	& \hspace{0.5cm} \leq \Big( \int_{|\bw|> R_w(t)} |\bw|^{2l_1} \mu_t(\di \bx,\di\bw) \Big)^{\frac{1}{l_1}} \cdot  \Big( \int_{|\bw|> R_w(t)}  \mu_t(\di \bx,\di\bw) \Big)^{1-\frac{1}{l_1}} \\
        & \hspace{0.5cm} \le \sqrt{\mathcal{M}_p(\mu_0, 2l_1) \mathcal{M}_{p}(\mu_0,D)} (1 + t)^{-\frac{(l_1-1)D \delta_p}{l_1}}.
\end{aligned}
\end{align}
In \eqref{C-12}, we combine all the estimates in \eqref{C-13} and \eqref{C-144} to find 
\begin{align*}
\begin{aligned}
		{\mathcal I}_{112} &\le 2 C_3(1 + t)^{-3\delta_p}  \\
		&\times \Big(  {\mathcal M}_{p}(\mu_0, 2) {\mathcal M}_p(\mu_0,D) (1 + t)^{-D\delta_p} +  \sqrt{\mathcal{M}_p(\mu_0, 2l_1) \mathcal{M}_{p}(\mu_0,D)} (1 + t)^{-\frac{(l_1-1)D \delta_p}{l_1}} \Big) \\
		&= 2C_3  \Big(  {\mathcal M}_{p}(\mu_0, 2) {\mathcal M}_p(\mu_0,D) (1 + t)^{-( (D + 3)\delta_p )} \\
		& \hspace{1.5cm} +  \sqrt{\mathcal{M}_p(\mu_0, 2l_1) \mathcal{M}_{p}(\mu_0,D)} (1 + t)^{- ( \delta_p (\frac{(l_1-1)D}{l_1} + 3) )} \Big) \\
		&\lesssim \underbrace{4C_3 \max \Big \{  {\mathcal M}_{p}(\mu_0, 2) {\mathcal M}_p(\mu_0,D),~ \sqrt{\mathcal{M}_p(\mu_0, 2l_1) \mathcal{M}_{p}(\mu_0,D)} \Big \} }_{=: \frac{C_6}{2}} (1 + t)^{-  \delta_p \left(\frac{(l_1-1)D}{l_1}+ 3\right)}.
	\end{aligned}
\end{align*}
\end{proof}
In \eqref{C-7}, we combine all the estimates in Lemma \ref{L4.3} and Lemma \ref{L4.4} to find  
\[
\frac{\di{\mathcal L}}{\di t} \leq  -C_5(1 + t)^{-(\beta \gamma_p + 3 \delta_p)}  {\mathcal L}  +   C_6 (1 + t)^{-  \delta_p \left(\frac{(l_1-1)D}{l_1}+ 3\right)} +  C_4  (1+t)^{-\frac{D}{2} (\gamma_p -1)}.
\]
\subsubsection{Zero asymptotic convergence of  ${\mathcal L}$} \label{sec:4.1.2}
Now, we are ready to apply Lemma \ref{L2.6} with 
\[  (y, p, q) \quad \Longleftrightarrow \quad  \Big ({\mathcal L}, C_5 (1 + t)^{-(3\delta_p + \beta \gamma_p)},   C_6 (1 + t)^{-  \delta_p \left(\frac{(l_1-1)D}{l_1}+ 3\right)} +  C_4  (1+t)^{-\frac{D}{2} (\gamma_p -1)} \Big) \]
to get 
\begin{align}
\begin{aligned}  \label{C-155}
{\mathcal L}(t) &\leq {\mathcal L}(0) \exp \Big[ - C_5 \int_0^t  (1 + \tau)^{-(3\delta_p + \beta \gamma_p)} \di\tau \Big] \\
& + \exp \Big[ - C_5 \int_{\frac{t}{2}}^t   (1 + \tau)^{- ( 3\delta_p + \beta \gamma_p)}  \di\tau \Big]\\
&\times \int_0^{\frac{t}{2}}  \Big[
C_6 (1 + \tau)^{-  \delta_p \left(\frac{(l_1-1)D}{l_1}+ 3\right)} +  C_4  (1+\tau)^{-\frac{D}{2} (\gamma_p -1)} \Big] \di\tau \\
& + \Big[   C_6 (1 + t/2)^{-  \delta_p \left(\frac{(l_1-1)D}{l_1}+ 3\right)} +  C_4  (1+t/2)^{-\frac{D}{2} (\gamma_p -1)}        \Big] \\
&\times\int_{\frac{t}{2}}^t \exp \Big[ -C_5 \int_s^t (1 + \tau)^{-(3\delta_p + \beta \gamma_p)} \di\tau \Big ] \di s \\
&=: {\mathcal I}_{21} + {\mathcal I}_{22} + {\mathcal I}_{23}.
\end{aligned}
\end{align}
\begin{lemma} \label{L4.5}
There exists a positive constant $C_7$ such that 
\begin{align*}
\begin{aligned}
& (i)~{\mathcal I}_{21} \leq  {\mathcal L}(0) \exp \Big[ - \frac{C_5}{1-(3\delta_p +\beta \gamma_p)}  \Big((1 + t)^{1-(3\delta_p +\beta \gamma_p)}-1 \Big) \Big].  \\
& (ii)~{\mathcal I}_{22} \leq  \exp \Big[ -\frac{C_5}{1-( 3\delta_p  + \beta \gamma_p)} \Big] \Big( (1 + t)^{1- (3\delta_p  + \beta \gamma_p)}  -(1 + t/2)^{1-( 3\delta_p + \beta \gamma_p)} \Big)\\
& \hspace{0.6cm} \times  \frac{C_6}{1-   \delta_p (\frac{(l_1-1)D }{l_1} + 3)} \Big[  (1 + t/2)^{1-  \delta_p (\frac{(l_1-1)D }{l_1} + 3)}  -1  \Big] \cdot \frac{C_4}{1- \frac{D}{2} (\gamma_p -1) } \Big[ (1 + t/2)^{1- (\frac{D}{2} (\gamma_p -1))}  -1  \Big]. \\
& (iii)~{\mathcal I}_{23} \leq C_7 \Big[   C_6 (1 + t/2)^{-  \delta_p (\frac{(l_1-1)D }{l_1} + 3)} +  C_4  (1+t/2)^{-\frac{D}{2} (\gamma_p -1)}        \Big].
\end{aligned}
\end{align*}
\end{lemma}
\begin{proof}
\noindent (i)~Note that 
\[
\exp\Big[  - C_5 \int_0^t  (1 + \tau)^{-(3\delta_p + \beta \gamma_p)} \di\tau       \Big ] =  \exp \Big[ - \frac{C_5}{1-(3\delta_p + \beta \gamma_p)} \Big((1 + t)^{1-(3\delta_p +\beta \gamma_p)}-1 \Big) \Big].
\]
\vspace{0.2cm}

\noindent (ii)~ Note that 
\begin{align*}
\begin{aligned}
& \exp \Big[ - C_5 \int_{\frac{t}{2}}^t   (1 + \tau)^{- ( 3\delta_p + \beta \gamma_p)}  \di\tau \Big]  \\
& \hspace{1cm} = \exp \Big[-\frac{C_5}{1-( 3\delta_p  + \beta \gamma_p)} \Big( (1 + t)^{1- (3\delta_p  + \beta \gamma_p) } -(1 + t/2)^{1-( 3\delta_p + \beta \gamma_p)} \Big) \Big], \\
&  C_6 \int_0^{\frac{t}{2}}  (1 + \tau)^{-  \delta_p (\frac{(l_1-1)D }{l_1} + 3)}  \di\tau = \frac{C_6}{1-  \delta_p (\frac{(l_1-1)D }{l_1} + 3)} \Big[  (1 + t/2)^{1-  \delta_p (\frac{(l_1-1)D }{l_1} + 3)}  -1  \Big], \\
& C_4 \int_0^{\frac{t}{2}}   (1+\tau)^{-\frac{D}{2} (\gamma_p -1)} \di\tau = \frac{C_4}{1- \frac{D}{2} (\gamma_p -1) } \Big[ (1 + t/2)^{1- (\frac{D}{2} (\gamma_p -1))}  -1  \Big].
\end{aligned} 
\end{align*}
We collect the above estimates to find the desired estimates. 
\vspace{0.2cm}

\noindent (iii)~Note that 
\begin{align*}
\begin{aligned}
& \int_{\frac{t}{2}}^t \exp \Big[ -C_5 \int_s^t (1 + \tau)^{-(3\delta_p + \beta \gamma_p)} \di\tau \Big] \di s \\
& \hspace{1cm} =   \int_{\frac{t}{2}}^t  \exp \Big[- \frac{C_5}{1-(3\delta_p + \beta \gamma_p)} \left((1 + t)^{1-(3\delta_p + \beta \gamma_p)}-(1 + s)^{1-(3\delta_p + \beta \gamma_p)} \right)\Big] \di s \\
& \hspace{1cm} \leq C_7.
\end{aligned}
\end{align*}
\end{proof}
In \eqref{C-155}, we use estimates in Lemma \ref{L4.5} to find the time-decay estimate of ${\mathcal L}$. 
\begin{proposition} \label{P4.1}
Suppose that $l_1, D, \beta, \gamma_p$ and $\delta_p$ satisfy
\begin{equation} \label{C-1-15}
l_1>1,\quad D \geq 2l_1, \quad 0 \leq \beta < 1, \quad \gamma_p > 1, \quad 0 \leq \beta \gamma_p <1, \quad  0 <  \delta_p  < \frac{1-\beta \gamma_p}{3},
\end{equation}
and let $\mu \in L^{\infty}([0,T);\mathcal{P}(\mathbb{R}^{2d}))$ be a measure-valued solution to \eqref{RKCS} with the initial datum $\mu_0$ satisfying
\[  \int_{\mathbb{R}^{2d}}  \mu_0(\di \bx,\di \bw) = 1, \qquad {\mathcal M}_p(\mu_0, D) < \infty. \]
Then, the velocity alignment functional ${\mathcal L}$ satisfies 
\[
{\mathcal L}(t) \leq {\mathcal O}(1) (1 + t)^{-\min \left\{   \delta_p \left(\frac{(l_1-1)D }{l_1} + 3\right),~\frac{D}{2} (\gamma_p -1)  \right\}}, \quad t \gg 1.
\]
\end{proposition}
\begin{remark}
Note that for $\beta = 1$, the fourth and sixth conditions in \eqref{C-1-15} cannot be compatible.
\end{remark}
\subsection{Spatial cohesiveness}  \label{sec:4.2}
In this subsection, we show that the spatial cohesiveness holds:
\[
\sup_{0 \leq t < \infty} \int_{\mathbb{R}^{2d}} \Big |\bx(t)-\frac{\bw_c(0)}{{\tilde F}}t-\bx_c(0) \Big |^2\mu_0(\di \bx,\di \bw) < \infty.
\]
Next, we use the Cauchy-Schwarz inequality to derive
\begin{align}\label{X1}
	\begin{aligned}
		&\dfrac{\di}{\di t} \int_{\mathbb{R}^{2d}} \Big |\bx(t)-\frac{\bw_c(0)}{{\tilde F}}t-\bx_c(0) \Big |^2\mu_0(\di \bx,\di \bw)\\ 
                 & = 2 \int_{\mathbb{R}^{2d}} \Big \langle \bx(t)-\frac{\bw_c(0)}{{\tilde F}}t-\bx_c(0),  \frac{\bw(t)}{{\tilde F}}-\frac{\bw_c(0)}{{\tilde F}} \Big \rangle\mu_0(\di \bx,\di \bw)\\ 
		&\le2\left(\int_{\mathbb{R}^{2d}} \Big |\bx(t)-\frac{\bw_c(0)}{{\tilde F}}t-\bx_c(0) \Big |^2\mu_0(\di \bx,\di \bw)\right)^{\frac{1}{2}}\left(\int_{\mathbb{R}^{2d}} \Big  |\frac{\bw(t)}{{\tilde F}}-\frac{\bw_c(0)}{{\tilde F}} \Big |^2\mu_0(\di \bx,\di \bw)\right)^{\frac{1}{2}}\\
		&\le2\left(\int_{\mathbb{R}^{2d}} |\bx(t)-\frac{\bw_c(0)}{{\tilde F}}t-\bx_c(0)|^2\mu_0(\di \bx,\di \bw)\right)^{\frac{1}{2}}\left(\int_{\mathbb{R}^{2d}} |\bw(t)-\bw_c(0)|^2\mu_0(\di \bx,\di \bw)\right)^{\frac{1}{2}}.
	\end{aligned}
\end{align}
This implies 
\begin{align}\label{X2}
	\begin{aligned}
& \dfrac{\di}{\di t}\left(\int_{\mathbb{R}^{2d}} |\bx(t)-\frac{\bw_c(0)}{{\tilde F}}t-\bx_c(0) |^2\mu_0(\di \bx,\di \bw)\right)^{\frac{1}{2}} \\
&\qquad\qquad\le\left(\int_{\mathbb{R}^{2d}}|\bw(t)-\bw_c(0)|^2\mu_0(\di \bx,\di \bw)\right)^{\frac{1}{2}} = \sqrt{{\mathcal L}(t)}.
	\end{aligned}
\end{align}
Then, we integrate the above relation from $0$ to $t$ to find 
\begin{align*}
\begin{aligned} \label{C-14}
& \left(\int_{\mathbb{R}^{2d}} |\bx(t)-\frac{\bw_c(0)}{{\tilde F}}t-\bx_c(0) |^2\mu_0(\di \bx,\di \bw)\right)^{\frac{1}{2}} \\
& \qquad\quad\leq \left(\int_{\mathbb{R}^{2d}} |\bx-\bx_c(0) |^2\mu_0(\di \bx,\di \bw)\right)^{\frac{1}{2}} + \int_0^t  \sqrt{{\mathcal L}(s)} \di s.
\end{aligned}
\end{align*}
Thus, spatial cohesiveness depends on the integrability of $\sqrt{\mathcal{L}}$.\newline 

Note that Proposition \ref{P4.1} implies
\[
\sqrt{{\mathcal L}(t)} \lesssim (1 + t)^{-\frac{1}{2} \min \left\{  \delta_p \left(\frac{(l_1-1)D}{l_1} + 3\right),~\frac{D}{2} (\gamma_p -1)  \right\}}, \quad t \gg 1.
\]
If 
\[ \min \left\{  \delta_p \left(\frac{(l_1-1)D}{l_1} + 3\right),~\frac{D}{2} (\gamma_p -1)  \right\} > 2, \]
then $\sqrt{{\mathcal L}(t)}$ is integrable and we have the desired spatial cohesive estimate. 

\section{Emergence of weak flocking II}\label{sec:5}
\setcounter{equation}{0}
In this section, we provide the emergence of weak flocking estimates for the exponentially decaying distributions. Again, we first introduce gauge functions:
\begin{align}
\begin{aligned}  \label{E-0}
& R_x(t) :=  R^0_x +  \gamma_e t \quad \mbox{for some}~~\gamma_e > c, \\
& R_{w}(t) :=   R_{w}^0(1 + t)^{\delta_e}, \quad \mbox{for some $\delta_e > 0$}, \\
&  \Omega_{\mathrm{eff}}(t) := \Big \{ (\bx, \bw) \in \bbr^{2d}:~|\bx| \leq R_x(t), \quad |\bw| \leq R_w(t) \Big \}.
\end{aligned}
\end{align}
Note that compared to polynomially decaying functions, $R_x$ has a linear growth. Since most analysis in the previous section can be done for our current case, we only point out some differences, whenever we need to mention them.

\subsection{Velocity alignment} \label{sec:5.1} 
In this subsection, we provide several preparatory estimates for the decay of ${\mathcal L}$. First, we provide a series of lemmas.
\begin{lemma} \label{L5.1}
For $\alpha>0$, let $\mu \in L^{\infty}([0,T);\mathcal{P}(\mathbb{R}^{2d}))$ be a measure-valued solution to \eqref{RKCS} with initial datum $\mu_0$ satisfying
\[  \int_{\mathbb{R}^{2d}}  \mu_0(\di \bx,\di \bw) = 1, \quad  {\mathcal M}_e(\mu_0, \alpha)  < \infty. \]
Then, there exist positive constants ${\tilde C}_1 = {\tilde C}_1(c, {\mathcal M}_e(\mu_0,\alpha), R_x^0, \alpha)$ and ${\tilde C}_2 = {\tilde C}_2(\beta, R_x^0, \gamma_e)$  such that 
\begin{align}\label{E-0-1}
	\int_{|\bx| \geq R_x(t)} \mu_t(\di \bx, \di \bw) \leq {\tilde  C}_1e^{-\alpha(\gamma_e -c)t}, \quad {\tilde C}_2 (1+t)^{-\beta} \leq  \underline{\phi} (t) \leq 1 \quad \mbox{for}~ t > 0.
\end{align}
\end{lemma}
\begin{proof}
\noindent (i)~ It follows from Corollary \ref{C2.2} and \eqref{E-0} that 
\begin{align}
\begin{aligned} \label{E-1}
& \int_{|\bx |\ge R_x(t)} \mu_t(\di \bx, \di \bw) \\
& \hspace{1cm} \le \mathcal{M}_{e}(\mu_0, \alpha) e^{\alpha (ct-R_x(t))} = \mathcal{M}_{e}(\mu_0, \alpha) e^{-\alpha R_x^0} e^{-\alpha(\gamma_e -c)t} \to 0, \quad \mbox{as $t \to \infty$}. 
\end{aligned}
\end{align}
\noindent (ii)~Consider two particle trajectories inside the effective domain:
\[ |\bx(t)|\le R_x(t), \quad |\bx_{\star}(t)| \leq R_x(t). \]
Then, we have
\begin{align*}
\begin{aligned}
& | \bx(t)-\bx_{\star}(t)|\le 2 R_x(t) \quad \mbox{and} \\
&  \underline{\phi} (t) :=  \inf\limits_{\bx(t),\bx_{\star}(t) \in B_{R_x(t)}} \phi(|\bx(t)-\bx_{\star}(t)|) \\
& \hspace{1cm} \ge \phi(2R_x(t))  = \frac{1}{  \Big(1+  4 R_x^2(t)  \Big )^{\frac{\beta}{2}}} \gtrsim \frac{1}{ |\max\{1 + 8 |R_x^0|^2,  8 \gamma_e^2 \}|^{\frac{\beta}{2}} } (1+t)^{-\beta},
\end{aligned}
\end{align*}
where we use the relation:
\[ \Big(1+  4 R_x^2(t)  \Big )^{\frac{\beta}{2}} =   \Big( 1 + 4 (R^0_x +  \gamma_e t)^2 \Big)^{\frac{\beta}{2}}  \lesssim  |\max\{1 + 8 |R_x^0|^2,  8 \gamma_e^2 \}|^{\frac{\beta}{2}} (1 + t)^{\beta}.    \]
\end{proof}
\begin{remark}
	Note that the exponential decay of 
	\[	\int_{|\bx| \geq R_x(t)} \mu_t(\di \bx, \di \bw)\]
	in \eqref{E-0-1} is different from the case of the polynomial decay distributions.
\end{remark}
\begin{lemma} \label{L5.2}
For $\alpha>0$, let $\mu \in L^{\infty}([0,T);\mathcal{P}(\mathbb{R}^{2d}))$ be a measure-valued solution to \eqref{RKCS} with initial datum $\mu_0$ satisfying
\[  \int_{\mathbb{R}^{2d}}  \mu_0(\di \bx,\di \bw) = 1, \quad {\mathcal M}_e(\mu_0, \alpha) = \int_{\mathbb{R}^{2d}} e^{\alpha\left(|\bx|+ |\bw|\right)} \mu_0(\mathrm{\di}\bx, \mathrm{\di }\bw) < \infty, \]
and let $\bw(t)$ and $\bw_{\star}(t)$ be particle trajectories such that 
\[  |\bw(t) |\le R_{w}(t), \quad  |\bw_{\star}(t)| \le R_{w}(t). \]
Then, there exists a positive constant ${\tilde C}_3 =  {\tilde C}_3 (R_w^0)$ such that 
\[ \left\langle \bw_{\star}(t)- \bw(t),\frac{\bw_{\star}(t)}{{\tilde F}_{\star}} - \frac{\bw(t)}{{\tilde F}} \right\rangle  \geq {\tilde C}_3 (1 + t)^{-3\delta_e} |\bw_{\star}(t)- \bw(t)|^2. \]
\end{lemma}
\begin{proof} Once we change $\delta_p$ with $\delta_e$, the proof will be exactly the same as in Lemma \ref{L4.2}. Hence we omit the details. 
\end{proof}
\subsubsection{Derivation of differential inequality for ${\mathcal L}$} \label{sec:5.1.2}
Note that  
\begin{align}
\begin{aligned} \label{E-2}
\frac{\di}{\di t} {\mathcal L}(t) &\leq  -\underline{\phi}(t) \int_{\mathbb{R}^{4d}}\left\langle \bw_{\star}(t)- \bw(t),\frac{\bw_{\star}(t)}{{\tilde F}_{\star}} - \frac{\bw(t)}{{\tilde F}} \right\rangle \mu_0(\di \bx,\di \bw) \mu_0(\di \bx_{\star},\di\bw_{\star})\\
&+ \underline{\phi}(t)  \int_{\mathbb{R}^{2d}} \int_{|\bx_{\star}(t)|> R_x(t)}\left\langle \bw_{\star}(t)-\bw(t),\frac{\bw_{\star}(t)}{{\tilde F}_{\star}} - \frac{\bw(t)}{{\tilde F}} \right\rangle \mu_0(\di \bx,\di \bw) \mu_0(\di \bx_{\star},\di\bw_{\star})\\
&+\underline{\phi}(t) \int_{|\bx(t)|> R_x(t)}\int_{\mathbb{R}^{2d}}\left\langle \bw_{\star}(t)-\bw(t),\frac{\bw_{\star}(t)}{{\tilde F}_{\star}} - \frac{\bw(t)}{{\tilde F}} \right\rangle \mu_0(\di \bx,\di \bw) \mu_0(\di \bx_{\star},\di\bw_{\star})\\
&=:{\mathcal I}_{31} + {\mathcal I}_{32} + {\mathcal I}_{33}.
\end{aligned}
\end{align}
In the following two lemmas, we estimate the term ${\mathcal I}_{3i},~i=1,2,3$ one by one. 
\begin{lemma}  \label{L5.3} 
There exists a positive constant $C_4 = C_4 (c, \mathcal{M}_p(\mu_0,2), C_1)$ such that 
\[ |{\mathcal I}_{32}(t)| \leq  \frac{{\tilde C}_4}{2} e^{-\frac{\alpha}{2}(\gamma_e -c)t}, \quad  |{\mathcal I}_{33}(t)| \leq  \frac{{\tilde C}_4}{2} e^{-\frac{\alpha}{2}(\gamma_e -c)t}. \]
\end{lemma}
\begin{proof}
By symmetry, we have
\[ {\mathcal I}_{32} = {\mathcal I}_{33}. \]
Hence, we can use the relation:
\[ {\mathcal M}_e(\mu_0, \alpha) < \infty  \quad \Longrightarrow \quad \mathcal{M}_p(\mu_0,2) < \infty, \]
and estimate in Lemma \ref{L5.1} to derive our estimates as we did for polynomial decay distributions. More precisely, we have
\begin{align*}
\begin{aligned}
{\mathcal I}_{32}(t) &\le 2c \underline{\phi}(t)  \left(\mathcal{M}_p(\mu_0,2)\right)^{\frac{1}{2}} \left[ \left(\int_{|\bx_{\star}(t)|> R_x(t)}\mu_0(\di \bx_{\star},\di\bw_{\star})\right)^{\frac{1}{2}}+\int_{|\bx_{\star}(t)|> R_x(t)}\mu_0(\di \bx_{\star},\di\bw_{\star})\right ] \\
& \leq  2c \left(\mathcal{M}_p(\mu_0,2)\right)^{\frac{1}{2}}  \Big( \sqrt{{\tilde C}_1} e^{-\frac{\alpha}{2}(\gamma_e -c)t}  +  {\tilde  C}_1 e^{-\alpha(\gamma_e -c)t} \Big) \\
&\le  4c \left(\mathcal{M}_p(\mu_0,2)\right)^{\frac{1}{2}} \max\{ \sqrt{{\tilde C}_1}, {\tilde C}_1 \}  e^{-\frac{\alpha}{2}(\gamma_e -c)t}    \\
&=: \frac{{\tilde C}_4}{2} e^{-\frac{\alpha}{2}(\gamma_e -c)t}.
\end{aligned}
\end{align*}
The same analysis can be done for ${\mathcal I}_{33}$. Hence, we omit its details. 
\end{proof}
\noindent Next, we return to the estimate of the term ${\mathcal I}_{31}$. It follows from Lemma \ref{L5.1} and Lemma \ref{L5.2} that 
\begin{align} 
\begin{aligned} \label{E-3}
{\mathcal I}_{31} & \le - {\tilde C}_3 \underline{\phi}(t) (1 + t)^{-3\delta_e} \int_{|\bw(t)|\le R_{w}(t)}\int_{|\bw_{\star}(t)|\le R_{w}(t)} | \bw_{\star}(t)-\bw(t)|^2\mu_0(\di \bx,\di \bw) \mu_0(\di \bx_{\star},\di\bw_{\star}) \\
&=-{\tilde C}_3 \underline{\phi}(t)  (1 + t)^{-3\delta_e} \int_{\mathbb{R}^{4d}} |\bw_{\star}(t)-\bw(t)|^2\mu_0(\di \bx,\di \bw) \mu_0(\di \bx_{\star},\di\bw_{\star})\\
& \hspace{0.2cm} + {\tilde C}_3 \underline{\phi}(t)  (1 + t)^{-3\delta_e} \int_{|\bw(t)|> R_{w}(t)}\int_{\mathbb{R}^{2d}} |\bw_{\star}(t)-\bw(t)|^2\mu_0(\di \bx,\di \bw) \mu_0(\di \bx_{\star},\di\bw_{\star})\\
&\hspace{0.2cm} +{\tilde C}_3 \underline{\phi}(t)  (1 + t)^{-3\delta_e} \int_{|\bw_{\star}(t)|> R_{w}(t)}\int_{\mathbb{R}^{2d}} |\bw_{\star}(t)-\bw(t)|^2\mu_0(\di \bx,\di \bw) \mu_0(\di \bx_{\star},\di\bw_{\star})\\
&=: {\mathcal I}_{311}+ {\mathcal I}_{312}+ {\mathcal I}_{313}.
\end{aligned}
\end{align}
In the following lemma, we estimate the term ${\mathcal I}_{31i},~i=1,2,3$ one by one. 
\begin{lemma} \label{L5.4}
There exist positive constants ${\tilde C}_5$ and ${\tilde C}_6$ such that 
\begin{align*}
\begin{aligned}
& (i)~{\mathcal I}_{311} =  -{\tilde C}_5(1 + t)^{-(\beta + 3 \delta_e)} {\mathcal L}. \\
& (ii)~{\mathcal I}_{312}  \leq \frac{{\tilde C}_6}{2} (1 + t)^{-3 \delta_e} e^{-\frac{\alpha}{2} (1+t)^{\delta_e}}. \\
& (iii)~{\mathcal I}_{313} \leq  \frac{{\tilde C}_6}{2} (1 + t)^{-3 \delta_e} e^{-\frac{\alpha}{2} (1+t)^{\delta_e}}. \\
& (iv)~{\mathcal I}_{31} \leq    -{\tilde C}_5(1 + t)^{-(\beta + 3 \delta_e)} {\mathcal L} + {\tilde C}_6 (1 + t)^{-3 \delta_e} e^{-\frac{\alpha}{2} (1+t)^{\delta_e}}.
\end{aligned}
\end{align*}
\end{lemma}
\begin{proof} (i) Similar to Lemma \ref{L4.4}, one has 
\[
{\mathcal I}_{311}  = -2  {\tilde C}_3 \underline{\phi}(t) (1 + t)^{-3\delta_e} {\mathcal L} \leq -2  {\tilde C}_2 {\tilde C}_3 (1 + t)^{-(\beta + 3 \delta_e)} {\mathcal L} = -{\tilde C}_5(1 + t)^{-(\beta + 3 \delta_e)} {\mathcal L}.
\]
\noindent (ii)~By symmetry, it is easy to see that 
\[ {\mathcal I}_{312} = {\mathcal I}_{313}. \]
For $	{\mathcal I}_{312}$, we have 
	\begin{align} 
		\begin{aligned} \label{E-4}
			{\mathcal I}_{312}&=  C_3 \underline{\phi}(t)  (1 + t)^{-3\delta_e}  \int_{|\bw|> R_w(t)}\int_{\mathbb{R}^{2d}} |\bw_{\star}(t)-\bw(t)|^2\mu_0(\di \bx,\di \bw) \mu_0(\di \bx_{\star},\di\bw_{\star}) \\
			&	\le 2 {\tilde C}_3(1 + t)^{-3\delta_e}   \int_{|\bw(t)|> R_w(t)}\int_{\mathbb{R}^{2d}} \Big(|\bw_{\star}(t)|^2+ |\bw(t)|^2 \Big) \mu_0(\di \bx,\di \bw) \mu_0(\di \bx_{\star},\di\bw_{\star}).
	\end{aligned}
	\end{align}	
By the same analysis as in Lemma \ref{L4.4}, we choose $l_1=2$ to see
\begin{align}
\begin{aligned} \label{E-5} 
& \int_{|\bw(t)|> R_w(t)}\int_{\mathbb{R}^{2d}} | \bw_{\star}(t)|^2\mu_0(\di \bx,\di \bw) \mu_0(\di \bx_{\star},\di\bw_{\star}) \leq    {\mathcal M}_{p}(\mu_0, 2) {\tilde C}_1 e^{-\alpha(1+t)^{\delta_e}}, \\
& \int_{|\bw(t)|> R_w(t)}\int_{\mathbb{R}^{2d}} | \bw(t)|^2\mu_0(\di \bx,\di \bw) \mu_0(\di \bx_{\star},\di\bw_{\star}) \leq \sqrt{\mathcal{M}_p(\mu_0, 4) {\tilde C}_1}  e^{-\frac{\alpha}{2} (1+t)^{\delta_e}}.
\end{aligned}
\end{align}
In \eqref{E-3}, we combine all the estimates in \eqref{E-4} and \eqref{E-5} to find 
\begin{align*}
\begin{aligned}
		{\mathcal I}_{31} &\le -{\tilde C}_5(1 + t)^{-(\beta + 3 \delta_e)} {\mathcal L}\\
		&\quad+ 2 {\tilde C}_3(1 + t)^{-3\delta_e} \Big( {\mathcal M}_{p}(\mu_0, 2) {\tilde C}_1 e^{-\alpha(1+t)^{\delta_e}} + \sqrt{\mathcal{M}_p(\mu_0, 4) {\tilde C}_1}  e^{-\frac{\alpha}{2} (1+t)^{\delta_e}} \Big) \\
		&\leq -{\tilde C}_5(1 + t)^{-(\beta + 3 \delta_e)} {\mathcal L}+{\tilde C}_6 (1 + t)^{-3 \delta_e} e^{-\frac{\alpha}{2} (1+t)^{\delta_e}}, \quad t > 0.
	\end{aligned}
\end{align*}
\end{proof}
In \eqref{E-2}, we combine all the estimates in Lemma \ref{L5.3} and Lemma \ref{L5.4} to find  
\begin{align*}
\begin{aligned}
\frac{\di {\mathcal L}}{\di t} &\leq   -{\tilde C}_5(1 + t)^{-(\beta + 3 \delta_e)} {\mathcal L} + {\tilde C}_6 (1 + t)^{-3 \delta_e} e^{-\frac{\alpha}{2} (1+t)^{\delta_e}} + {\tilde C}_4 e^{-\frac{\alpha}{2}(\gamma_e -c)t} \\
& \leq -{\tilde C}_5(1 + t)^{-(\beta + 3 \delta_e)} {\mathcal L} + {\tilde C}_7 e^{-\frac{\alpha}{2}(1+t)^{\delta_e}}.
\end{aligned}
\end{align*}
\subsubsection{Zero asymptotic convergence of  ${\mathcal L}$} \label{sec:5.1.3}
Now, we apply Lemma \ref{L2.6} with 
\[  (y, p, q) \quad \Longleftrightarrow \quad  \Big ({\mathcal L}, -{\tilde C}_5(1 + t)^{-(\beta + 3 \delta_e)},~{\tilde C}_7 e^{-\frac{\alpha}{2}(1+t)^{\delta_e}}   \Big) \]
to get 
\begin{align}
\begin{aligned}  \label{E-6}
{\mathcal L}(t) &\leq {\mathcal L}(0) \exp \Big[  - {\tilde C}_5 \int_0^t  (1 + \tau)^{-(3\delta_e + \beta)} \di\tau \Big] \\
& + {\tilde C}_7  \exp \Big[ - {\tilde C}_5 \int_{\frac{t}{2}}^t   (1 + \tau)^{- ( 3\delta_e + \beta)}  \di\tau \Big] \int_0^{\frac{t}{2}} \exp \Big[-\frac{\alpha}{2}(1+\tau)^{\delta_e} \Big] \di\tau \\
& +  {\tilde C}_7 e^{-\frac{\alpha}{4}(1+t)^{\delta_e}}   \int_{\frac{t}{2}}^t \exp \Big[-{\tilde C}_5 \int_s^t (1 + \tau)^{-(3\delta_e + \beta)} \di\tau \Big] \di s \\
&=: {\mathcal I}_{41} + {\mathcal I}_{42} + {\mathcal I}_{43}.
\end{aligned}
\end{align}
\begin{lemma} \label{L5.5}
There exists a positive constant ${\tilde C}_8$ such that 
\begin{align*}
\begin{aligned}
& (i)~{\mathcal I}_{41} \leq  {\mathcal L}(0) \exp \Big[- \frac{{\tilde C}_5}{1-(3\delta_e + \beta)} \Big ((1 + t)^{1-(3\delta_e +\beta)}  -1\Big) \Big].  \\
& (ii)~{\mathcal I}_{42} \leq  {\tilde C}_8 \exp \Big[- \frac{{\tilde C}_5 t}{2\left(1+t\right)^{3\delta_e + \beta}} \Big]. \\
& (iii)~{\mathcal I}_{43} \leq   {\tilde C}_8 \exp \Big[ -\frac{\alpha}{4}(1+t)^{\delta_e} \Big].
\end{aligned}
\end{align*}
\end{lemma}
\begin{proof}
\noindent (i)~Note that 
\[
\exp\Big[ - {\tilde C}_5 \int_0^t  (1 + \tau)^{-(3\delta_e + \beta)} \di\tau \Big] =  \exp \Big[ - \frac{{\tilde C}_5}{1-(3\delta_e + \beta)} \Big((1 + t)^{1-(3\delta_e +\beta)}-1 \Big) \Big].
\]
\vspace{0.2cm}

\noindent (ii)~ Note that 
\begin{align}
\begin{aligned} \label{E-7}
& \exp \Big[- {\tilde C}_5 \int_{\frac{t}{2}}^t   (1 + \tau)^{- ( 3\delta_e + \beta)}  \di\tau \Big]  \\
& \hspace{2cm} = \exp\Big[ -\frac{{\tilde C}_5}{1-( 3\delta_e  + \beta)} \Big( (1 + t)^{1- (3\delta_e  + \beta)} -(1 + t/2)^{1-( 3\delta_e + \beta)} \Big) \Big], \\
&  \int_0^{\frac{t}{2}} e^{-\frac{\alpha}{2}(1+\tau)^{\delta_e}} \di\tau \le {\tilde C}_8,\quad\delta_e>0.
\end{aligned} 
\end{align}
Then we compute the derivative of function $(1+x)^{p},~p \in (0,1)$ as follows: \[ \frac{\di}{\di x}(1+x)^{p}= \frac{p}{(1+x)^{1-p}}. \]
Next, we use this relation to find some $\xi\in[0,t]$ such that
\begin{align}
\begin{aligned} \label{E-8}
& (1 + t/2 + t/2)^{1-(3\delta_e + \beta)} -(1 + t/2)^{1-(3\delta_e + \beta)}  \\
& \hspace{0.5cm} =(1-3\delta_e + \beta)\frac{1}{\left(1+\xi\right)^{3\delta_e + \beta}} (1 + t/2 + t/2-1-t/2)\\
& \hspace{0.5cm}  \geq (1-3\delta_e + \beta)\frac{1}{\left(1+t\right)^{3\delta_e + \beta}}\frac{t}{2} \\
& \hspace{0.5cm}  =(1-3\delta_e + \beta)\frac{t}{2\left(1+t\right)^{3\delta_e + \beta}}.
\end{aligned}
\end{align}
Now, we combine \eqref{E-7} and \eqref{E-8} to get 
\[
\exp \Big[ - {\tilde C}_5 \int_{\frac{t}{2}}^t   (1 + \tau)^{- ( 3\delta_e + \beta)}  \di\tau \Big] \leq \exp\Big [ - {\tilde C}_5 \frac{t}{2\left(1+t\right)^{3\delta_e + \beta}}\Big], \quad0< 3\delta_e+\beta<1.
\]
We collect the above estimates to find the desired estimates. \newline

\noindent (iii)~Note that 
\[
 \int_{\frac{t}{2}}^t \exp \Big[ -{\tilde C}_5 \int_s^t (1 + \tau)^{-(3\delta_e + \beta)} \di\tau \Big] \di s \leq {\tilde C}_8. 
\]
\end{proof}
In \eqref{E-6}, we combine all the estimates in Lemma \ref{L5.5} to get the decay of ${\mathcal L}$. 
\begin{proposition}
Suppose that $\alpha, \beta, \gamma_e$ and $\delta_e$ satisfy
\begin{equation} \label{E-15}
\alpha > 0, \quad 0 \leq \beta < 1, \quad \gamma_e > c, \quad 0 < \delta_e < \frac{1-\beta}{3},
\end{equation}
and let $\mu \in L^{\infty}([0,T);\mathcal{P}(\mathbb{R}^{2d}))$ be a measure-valued solution to \eqref{RKCS} with initial datum $\mu_0$ satisfying
\[  \int_{\mathbb{R}^{2d}}  \mu_0(\di \bx,\di \bw) = 1, \qquad {\mathcal M}_e(\mu_0, \alpha) < \infty.  \]
Then, the velocity alignment functional ${\mathcal L}$ satisfies 
\begin{align}\label{E-16}
{\mathcal L}(t) \lesssim {\tilde C}_{10} \left(e^{-{\tilde C}_9 t^{1-(3\delta_e + \beta)}}+e^{-\frac{\alpha}{4}(1+t)^{\delta_e}} \right),
\end{align}
for some positive constants ${\tilde C}_9$ and ${\tilde C}_{10}$.
\end{proposition}
\begin{remark}
Note that for $\beta = 1$, the fourth conditions in \eqref{E-15} cannot be compatible.
\end{remark}
\subsection{Spatial cohesiveness}  \label{sec:5.2}
In this subsection, we show that the spatial cohesiveness holds:
\[
\sup_{0 \leq t < \infty} \int_{\mathbb{R}^{2d}} \Big |\bx(t)-\frac{\bw_c(0)}{{\tilde F}}t-\bx_c(0) \Big |^2\mu_0(\di \bx,\di \bw) < \infty.
\]
As in Section \ref{sec:4.2}, we can derive 
\begin{align*}
& \left(\int_{\mathbb{R}^{2d}} |\bx(t)-\frac{\bw_c(0)}{{\tilde F}}t-\bx_c(0) |^2\mu_0(\di \bx,\di \bw)\right)^{\frac{1}{2}} \\
&\qquad\qquad \leq \left(\int_{\mathbb{R}^{2d}} |\bx-\bx_c(0) |^2\mu_0(\di \bx,\di \bw)\right)^{\frac{1}{2}} + \int_0^t  \sqrt{{\mathcal L}(s)} \di s.
\end{align*}
Thus, spatial cohesiveness depends on the integrability of $\sqrt{{\mathcal L}}$. Note that the estimate \eqref{E-16} implies
\[
\sqrt{{\mathcal L}(t)} \lesssim \sqrt{{\tilde C}_{10} }\left(e^{-\frac{{\tilde C}_9}{2} t^{1-(3\delta_e + \beta)}}+e^{-\frac{\alpha}{8}(1+t)^{\delta_e}} \right), \quad t \geq 0.
\]
Then $\sqrt{{\mathcal L}(t)}$ is integrable and we have the desired spatial cohesive estimate.

\section{Conclusion}\label{sec:6}
\setcounter{equation}{0}
In this paper, we have derived weak flocking estimates for the relativistic kinetic Cucker-Smale model in a fully non-compact support setting. More precisely, we established a framework for weak flocking behavior and rigorously demonstrated polynomial and exponential convergence towards weak flocking under polynomial and exponential decay distributions, respectively. These emergent dynamics under Gaussian, sub-Gaussian, and \(D\)-moment integrable distributions provided a suitable basis for defining weak flocking behavior. In summary, our results contribute to the robust understanding of the RKCS model in non-compact support settings and lay a foundation for future exploration of relativistic collective behavior in unbounded physical domains. However, our proof technique does not extend to the critical exponent ($\beta = 1$). Additionally, the current work assumes all-to-all networks, and we leave the study of collective behavior under a general network topology as an open question. We leave these interesting issues as our further work.

\section*{Conflict of interest statement}
The authors declare no conflicts of interest.

\section*{Data availability statement}
The data supporting the findings of this study are available from the corresponding author upon reasonable request.
\section*{Ethical statement}
The authors declare that this manuscript is original, has not been published before, and is not currently being considered for publication elsewhere. The study was conducted by the principles of academic integrity and ethical research practices. All sources and contributions from others have been properly acknowledged and cited. The authors confirm that there is no fabrication, falsification, plagiarism, or inappropriate manipulation of data in the manuscript.

\bibliography{sn-bibliography}

\end{document}